\documentclass{article}[11 pt]
\usepackage[utf8]{inputenc}
\usepackage[T1]{fontenc}
\usepackage[english]{babel} 
\usepackage{amsmath, amssymb, amsthm}
\usepackage{tikz}
\usepackage{hyperref}
\usepackage{algorithm}
\usepackage{algorithmic}

\usepackage{authblk}

\newcommand{\R}{\mathbb{R}}
\newcommand{\N}{\mathbb{N}}
\newcommand{\cP}{\mathcal{P}}
\newcommand{\cN}{\mathcal{N}}

\newcommand{\Nh}{{\mathcal N}}

\newcommand{\V}{{\mathcal V}}
\newcommand{\M}{M}

\newtheorem{theorem}{Theorem}[section]
\newtheorem{assumption}[theorem]{Assumption}
\newtheorem{remark}[theorem]{Remark}
\newtheorem{lemma}[theorem]{Lemma}
\newtheorem{proposition}[theorem]{Proposition}
\newtheorem{definition}[theorem]{Definition}

\numberwithin{equation}{section}

\title{Reduced basis method for non-symmetric eigenvalue problems: application to the multigroup neutron diffusion equations}
\author[1]{Yonah Conjungo Taumhas}
\author[2]{Geneviève Dusson}
\author[3]{Virginie Ehrlacher}
\author[3]{Tony Lelièvre}
\author[1]{François Madiot}
\affil[1]{Université Paris-Saclay, CEA,  Service d'\'Etudes des Réacteurs et de Mathématiques Appliquées, 91191, Gif-sur-Yvette, France}
\affil[2]{Laboratoire de Mathématiques de
Besancon, UMR CNRS 6623,  
Universit\'e de Franche-Comt\'e, 16 route de Gray, 25030 Besan\c con, France}
\affil[3]{CERMICS, \'Ecole des Ponts ParisTech,  
  6-8 avenue Blaise Pascal, Cité Descartes, 77455 Marne-la-Vallée, Cedex 2, France}
\date{\today}

\begin{document}

\maketitle

\begin{abstract}
In this article, we propose a reduced basis method for parametrized non-symmetric eigenvalue problems arising in the loading pattern optimization of a nuclear core in neutronics. To this end, we derive {\it a posteriori} error estimates for the eigenvalue and left and right eigenvectors. 
The practical computation of these estimators requires the estimation of a constant called prefactor, which we can express as the spectral norm of some operator. We provide some elements of theoretical analysis which illustrate the link between the expression of the prefactor we obtain here and its well-known expression in the case of symmetric eigenvalue problems, either using the notion of numerical range of the operator, or via a perturbative analysis. Lastly, we propose a practical method in order to estimate this prefactor which yields interesting numerical results on actual test cases.
We provide detailed numerical simulations on two-dimensional examples including a multigroup neutron diffusion equation.
\end{abstract}

\section{Introduction}

In this work, we are interested in developing a numerical method to efficiently compute the solutions of a parametrized non self-adjoint eigenvalue problem for a large number of parameter values. 
An example of application where this type of problem occurs and which motivates the present work is the resolution of criticality problems in neutronics. The method we propose relies on a reduced basis technique~\cite{maday2002priori,prud2002reliable}.

Model order reduction methods such as reduced basis techniques~\cite{boyaval2010reduced, Hesthaven2016-xm,Quarteroni2015-db} are useful to accelerate the computation of approximate solutions to parameterized problems. 
In the context of neutronics, parametrized problems naturally occur when optimizing the loading pattern of a nuclear core~\cite{dechaine1995nuclear,do2007application,turinsky2005nuclear}.  Mathematically, this amounts to optimizing an objective function which involves the solution to a generalized non-symmetric eigenvalue problem, parameterized by the fuel assemblies distribution.
The objective of this work is thus to propose a reduced basis technique in this context. It can be seen as a generalization of~\cite{Horger2017-jc, Fumagalli2016-fi,Boffi2023-ir}, where reduced basis methods for symmetric eigenvalue problems have been developed. 
A main difficulty is to obtain reliable {\it a posteriori} estimators in order to build the reduced basis and certify the results obtained with the reduced problem. We refer to~\cite{Gedicke2009-rx,giani2016robust} for {\it a posteriori} error estimators for non self-adjoint eigenvalue problems in a classical finite element context. Let us also mention the recent work~\cite{Edel2022-rg} and references therein for efficient {\it a posteriori} estimators for non-symmetric problems. We refer to~\cite{buchan2013pod,german2019reduced,lorenzi2018adjoint,sartori2015reduced} for some other applications of model order reduction techniques applied to neutronics.

As mentioned above, the main bottleneck here is to propose efficient {\it a posteriori} error estimators for a reduced basis approximation of non self-adjoint eigenvalue problems. More precisely, we consider a situation where one is interested in computing the eigenvalue of smallest modulus of a parameterized eigenvalue problem, which is assumed to be simple. 
The {\it a posteriori} error estimators read as products of norms of the residuals of the direct and associated adjoint eigenvalue problems times a multiplicative constant, which we call hereafter the {\it prefactor}. 
Computing an accurate and optimal value of this prefactor is not an easy task, compared to the case of symmetric eigenvalue problems where it can be expressed by means of the spectral gap of the considered operator.

Three main contributions are proposed in this work. 
First, we derive an expression of the prefactor in the case of non-symmetric eigenvalue problems as the spectral norm of a composition of some well-chosen operators. 
Second, we provide some elements of theoretical analysis to illustrate the close link between the obtained expression of the prefactor and its well-known expression in the case of symmetric eigenvalue problems. 
This link is highlighted in two different ways: first, we give an expression of the prefactor using the distance between the approximate eigenvalue and the numerical range of the non-symmetric operator and observe that the numerical range plays a similar role as the spectrum of the operator in the self-adjoint case; 
second, we use perturbative arguments to give a second-order development of this prefactor when the operator is a small perturbation of a symmetric operator.  
As our third contribution, we propose a practical heuristic method to estimate the prefactor in the present reduced basis context and demonstrate the efficiency of the approach on test cases stemming from neutronics applications.

The outline of this article is as follows. In Section~\ref{sec:context}, we describe the prototypical reference problem of interest as well as the model order reduction method we use,  which relies on a greedy algorithm to build the reduced basis. The greedy procedure requires {\it a posteriori} estimators which are presented in Section~\ref{sec:apost}. These estimators are basically built as the products of residual norms with prefactors, whose computations are discussed in Section~\ref{sec:prefactor}.
Finally, we provide numerical results on two different examples in Section~\ref{sec:num}: a toy example of a two-dimensional two-group problem, and a two-dimensional simple model of a minicore.

\section{Reduced basis method for non-symmetric generalized eigenvalue problems}
\label{sec:context}

The objective of this section is to introduce the mathematical framework and the model order reduction method we consider. In Section~\ref{sec:high_fidelity}, we describe the reference high-fidelity generalized eigenvalue  problem of interest. The reduced-order model is then presented in Section~\ref{sec:reduced_basis}, and the greedy algorithm used to build the reduced basis is finally explained in Section~\ref{sec:greedy}.

\subsection{Reference high-fidelity problem}\label{sec:high_fidelity}

Let us present the parametrized generalized eigenvalue problem  for which we wish to build a reduced-order problem. Let $\Nh \in \mathbb{N}^*$ be a positive integer which is assumed to be large in our context. In all the following, $\mathbb{R}^\Nh$ is endowed with the Euclidean scalar product\footnote{It is easy to generalize the results presented below to any Hilbertian norm.} denoted by $\langle\cdot,\cdot\rangle$ and associated norm $\|\cdot\|$. For all values of the vector of parameters $\mu$ belonging to the set of parameter values $\mathcal{P} \subset \mathbb R^p$ for some $p\ge 1$, we consider two matrices $A_\mu$ and $B_\mu$ in $\mathbb{R}^{\Nh \times \Nh}$ and the following generalized eigenvalue problem:
Find $(u_{\mu}, \lambda_{\mu}) \in \mathbb{R}^{\Nh} \times \mathbb{C}$ such that $\lambda_\mu$ is an eigenvalue with minimal modulus:
\begin{align}
       A_{\mu}u_{\mu} = \lambda_{\mu} B_{\mu}u_{\mu},
       \quad \|u_{\mu}\| =1.
       \label{eq:hfp_lambda}
\end{align}

As is classical in the context of reduced basis methods, we refer to problem (\ref{eq:hfp_lambda}) as the {\it high-fidelity} (HF) problem. We make the following additional assumption which is satisfied in the problems we are eventually interested in for neutronics applications. 
\begin{assumption}\label{assumption:Perron-Frobenius}
    For any parameter $\mu \, \in \mathcal{P}$, $A_\mu$ is invertible and there exists a unique positive eigenvalue $\lambda_\mu$ which realizes the smallest modulus solution to~\eqref{eq:hfp_lambda}. Moreover, the eigenvalue $\lambda_\mu$ is simple.
\end{assumption}
A consequence of Assumption~\ref{assumption:Perron-Frobenius} is that $u_\mu$ is uniquely defined (up to a sign), $\lambda_\mu$ is real, and that there is a spectral gap between $\lambda_\mu$ and the other eigenvalues solutions to problem~\eqref{eq:hfp_lambda}, a property that will also play a role in the {\it a posteriori} analysis below.
The associated adjoint problem then reads: Find $({u}_{\mu}^*,\lambda_{\mu}) \in \mathbb{R}^{\Nh} \times \mathbb{R}_+^*$ such that
\begin{align}
           A_{\mu}^Tu_{\mu}^* = \lambda_{\mu} B_{\mu}^Tu_{\mu}^*, \quad
           \|u_{\mu}^*\| = 1.
           \label{eq:hfp_lambda_adjoint}
\end{align}
Let us mention here that, for any $A\in \R^{\Nh \times \Nh}$, the adjoint matrix $A^T\in \R^{\Nh\times \Nh}$ is defined relatively to the scalar product $\langle \cdot, \cdot \rangle$ as follows: 
$$
\forall u,v \in \R^\Nh, \quad \langle v, Au \rangle = \langle A^T v, u \rangle. 
$$
Similarly, for any column vector $u\in \R^\cN$, we denote by $u^T$ the unique line vector of $\R^\cN$ such that 
$$
\forall v \in \R^\cN, \quad u^T v = \langle u ,v \rangle. 
$$
From Assumption~\ref{assumption:Perron-Frobenius}, the  eigenvectors $u_\mu$ and $u^*_\mu$ can be chosen real. In practice, the solutions to~\eqref{eq:hfp_lambda} and~\eqref{eq:hfp_lambda_adjoint} are approximated by an inverse power method, which will be properly described in Algorithm~\ref{algorithm:power_it}. 

 We also define for all $\mu \in \mathcal P$ the so-called effective multiplication factor 
\[
    k_{\mu} := \frac{1}{\lambda_{\mu}}.
\]

There holds
\begin{equation}
    \label{eq:k_mu}
     k_{\mu} = \dfrac{\langle u_{\mu}^*,B_{\mu}u_{\mu} \rangle}{\langle u_{\mu}^*,A_{\mu}u_{\mu} \rangle}.
\end{equation}

\begin{remark}
On the one hand, Assumption~\ref{assumption:Perron-Frobenius} holds for instance if $A_{\mu}$ is invertible and the matrix $A_{\mu}^{-1}B_{\mu}$ coming from problem~\eqref{eq:hfp_lambda} satisfies the assumptions of the Perron--Frobenius theorem~\cite{AllaireDespresGolse}. Note that under the assumption that $A_{\mu}$ is invertible, $\lambda_\mu$ is solution to~\eqref{eq:hfp_lambda} if and only if  $ k_{\mu}$ is an eigenvalue associated with the matrix $A_{\mu}^{-1}B_{\mu}$. 
On the other hand, in the context of neutronics applications mentioned earlier and detailed in Section~\ref{sec:physical_context},~\eqref{eq:hfp_lambda} is obtained as an appropriate discretization of a continuous problem where the associated resolvent operator satisfies the assumptions of the Krein--Rutman theorem and thus admits a simple real greatest eigenvalue in modulus denoted $k^{\rm ex}_{\mu}$. Since $1/k^{\rm ex}_{\mu}$ is solution to the continuous problem, the
smallest eigenvalue of~\eqref{eq:hfp_lambda} in modulus is also {expected to be} simple and positive {for fine enough discretization, i.e. large enough $\Nh$}. 
\end{remark}

We now assume in all the rest of the article that Assumption~\ref{assumption:Perron-Frobenius} is satisfied.

\begin{lemma}
\label{lem:lem00}
Under Assumption~\ref{assumption:Perron-Frobenius},  $\displaystyle \langle u_{\mu}^*,A_{\mu} u_{\mu} \rangle \neq 0$.
\end{lemma}

We postpone the proof of this lemma after Lemma~\ref{lemma:properties_proj}.

\begin{remark}
Without Assumption~\ref{assumption:Perron-Frobenius}, it is possible that $\langle u_{\mu}^*,A_{\mu} u_{\mu}\rangle = 0$. 
Indeed, a simple example is to take $ A_{\mu}=\begin{pmatrix} 1&-1\\0&1\end{pmatrix}$ and $ B_{\mu}=\begin{pmatrix} 1&0\\0&1\end{pmatrix}$. Let $u_{\mu}=(1,0)$ and $u_{\mu}^*=(0,1)$. 
 Equations~\eqref{eq:hfp_lambda} and~\eqref{eq:hfp_lambda_adjoint} are satisfied with $\lambda_{\mu} = 1$ while $\langle u_{\mu}^*,A_{\mu} u_{\mu}\rangle = 0$. 
\end{remark}

We are interested in situations where one has to solve the reference high-fidelity problem~\eqref{eq:hfp_lambda} quickly and for many values of $\mu$. The idea is to build a reduced basis using some solutions of~\eqref{eq:hfp_lambda} (so-called snapshots) computed offline, and to use a Galerkin method to project the problem~\eqref{eq:hfp_lambda} onto this reduced basis, see Section~\ref{sec:reduced_basis}. This requires {\it a posteriori} estimators to {wisely} select the parameters used to build the reduced basis, as well as to certify the numerical results obtained online on the reduced basis: this is discussed in Section~\ref{sec:apost}.

\subsection{Reduced-order model}
\label{sec:reduced_basis}

The aim of this section is to present the reduced-order model obtained from a given reduced basis to get an approximation of (\ref{eq:hfp_lambda}). 
Let us consider a reduced linear subspace $\V$ of $\mathbb{R}^\Nh$ of dimension $N$ much smaller than~$\Nh$, built in such a way that that any solution of problem~\eqref{eq:hfp_lambda} can be accurately approximated by an element of the space $\V$ (the construction of such a subspace will be discussed in the next section). A reduced-order model for problem~\eqref{eq:hfp_lambda} can then be obtained from the reduced space  $\V$ as follows. Let $\left(\xi_i\right)_{1 \leqslant i \leqslant N}$ be an orthonormal basis of $\V$. 
The reduced matrices $A_{\mu, N}\in \R^{N\times N} , B_{\mu, N}\in \R^{N\times N}$ are defined as follows: for all $1\leq i,j \leq N$,
\begin{align*}
  A_{\mu, N}^{ij}& := \langle \xi_i, A_\mu \xi_j \rangle,\\
    B_{\mu, N}^{ij}& := \langle \xi_i, B_\mu \xi_j \rangle.
\end{align*}
The reduced-order model then consists in solving: Find $(c_{\mu,N},\lambda_{\mu,N}) \in \mathbb{R}^N \times \mathbb{C}$ such that $\lambda_{\mu,N}$ is an eigenvalue with smallest modulus to
\begin{align}
&            A_{\mu,N}c_{\mu,N} = \lambda_{\mu,N} B_{\mu,N}c_{\mu,N},\quad u_{\mu,N} = \sum_{i=1}^N c_{\mu,N}^i\xi_i,
\quad \text{and} \quad
\|u_{\mu,N}\| = 1,
\label{eq:reducedpb}
\end{align}
where for all $1\leq i \leq N$, $c_{\mu,N}^i$ is the $i^{th}$ component of the vector $c_{\mu,N}$.
Similarly as in Section~\ref{sec:high_fidelity} (see~Assumption~\ref{assumption:Perron-Frobenius}), we make the following assumption. 
\begin{assumption}\label{assumption:Perron-Frobenius_reduced_pb}
    For any parameter $\mu \, \in \mathcal{P}$, the matrix $A_{\mu,N}$ is invertible and there exists a unique positive eigenvalue $\lambda_{\mu,N}$ which realizes the smallest modulus solution to~\eqref{eq:reducedpb}. Moreover, the eigenvalue $\lambda_{\mu,N}$ is simple.
\end{assumption}
Under this assumption, $c_{\mu,N}$ and $u_{\mu,N}$ are uniquely defined up to a sign and $\lambda_{\mu,N}$ is real. Endowing the space $\R^N$ with the canonical Euclidean scalar product $\langle \cdot, \cdot \rangle_{\ell^2}$, we can consider the solution to the associated reduced adjoint problem: 
Find $(c_{\mu,N}^*,\lambda_{\mu,N}) \in \mathbb{R}^N \times \mathbb{R}_+^*$ such that the eigenvalue $\lambda_{\mu,N}$ is the smallest in modulus and
\begin{align}
&            A_{\mu,N}^t c_{\mu,N}^* = \lambda_{\mu,N} B_{\mu,N}^t c_{\mu,N}^*,
\quad u_{\mu,N}^* = \sum_{i=1}^N c_{\mu,N}^{*,i}\xi_i, \quad \text{and} \quad \|u_{\mu,N}^*\|=1.
\label{eq:reducedpb_adjoint}
\end{align}
where for all $1\leq i \leq N$, $c_{\mu,N}^{*,i}$ is the $i^{th}$ component of the vector $c_{\mu,N}^{*}$ and $A_{\mu,N}^t$ and $ B_{\mu,N}^t$ are respectively the transpose of the matrix $A_{\mu,N}$ and  $ B_{\mu,N}$.
Moreover, under this assumption, we have $\displaystyle \langle c_{\mu,N}^*,A_{\mu,N} c_{\mu,N} \rangle_{\ell^2} =  \langle u_{\mu,N}^*,A_{\mu} u_{\mu,N} \rangle \neq 0$ {(see Lemma~\ref{lem:lem00})}, 
and we define
\begin{equation}
    \label{eq:k_N}
     k_{\mu,N} = \dfrac{\langle c_{\mu,N}^*,B_{\mu,N}c_{\mu,N} \rangle_{\ell^2}}{\langle c_{\mu,N}^*,A_{\mu,N}c_{\mu,N} \rangle_{\ell^2}}
     = \dfrac{\langle u_{\mu,N}^*,B_{\mu}u_{\mu,N} \rangle}{\langle u_{\mu,N}^*,A_{\mu}u_{\mu,N} \rangle}. 
\end{equation}

In practice, we use the inverse power method to solve~\eqref{eq:reducedpb} and~\eqref{eq:reducedpb_adjoint}.  If both algorithms converge, we refer to the outputs~$c_{\mu,N}$ and $c^*_{\mu,N}$ as the right and left eigenvectors of the reduced problem. If one of the power methods does not converge, the reduced basis is enriched using the high-fidelity left and right eigenvectors for the considered parameter value, see the construction of the reduced space in the next section. Note that the power methods applied to~\eqref{eq:reducedpb} and~\eqref{eq:reducedpb_adjoint} (resp. to~\eqref{eq:hfp_lambda} and~\eqref{eq:hfp_lambda_adjoint}) are guaranteed to converge if  Assumption~\ref{assumption:Perron-Frobenius_reduced_pb} (resp. Assumption~\ref{assumption:Perron-Frobenius}) is satisfied.
In the numerical examples presented in Section~\ref{sec:num}, we observe that the two power methods (for the direct and adjoint reduced eigenvalue problems) indeed converge and that $\displaystyle \langle c_{\mu,N}^*,A_{\mu,N} c_{\mu,N} \rangle_{\ell^2} \neq 0$ as soon as the reduced space has a sufficiently large dimension (typically $N \ge 4$ is sufficient in the numerical results presented in Section~\ref{sec:num}).

\subsection{Choice of the reduced space: greedy algorithm}\label{sec:greedy}

In practice, the reduced space $\V$ used in the reduced-order model described in the previous section is built following the standard procedure of the reduced basis technique~\cite{maday2002priori}. We first 
initialize the reduced space~$\V_0$ as a very low-dimensional space spanned by the first modes obtained from a Proper Orthogonal Decomposition of a family of vectors composed of a few snapshots of the direct and adjoint problems. 
A sequence of parameter values $(\mu_n)_{n\geq 1}$ is then selected from a greedy procedure described below, from which nested reduced spaces $(\V_n)_{n\geq 1}$ are built 
as follows: 
\begin{equation}
    \forall n\geq 1, \quad \V_n = \text{Span}\left\{u_{\mu_1},\ldots,u_{\mu_n},u^*_{\mu_1},\ldots,u^*_{\mu_n}\right\}. \label{eq:VN_def}
\end{equation}
In the following, we denote by $N_n:= {\rm dim} \V_n$ and by $u_{\mu, N_n}$, $u_{\mu, N_n}^*$, $\lambda_{\mu,N_n}$ and $k_{\mu, N_n}$ the solutions of the reduced eigenvalue problems described in the previous section for $\V = \V_n$. 

\begin{remark}
 The choice made  in~\eqref{eq:VN_def} to enrich the sequence of reduced spaces with both the eigenvector of the direct and of the adjoint eigenvalue problem stems from the {\it a priori} error analysis of Galerkin approximations of generalized eigenvalue problems (see~\cite{babuvska1991eigenvalue,Boffi2010-wh}). Indeed, in an asymptotic regime, the error between the approximate and exact eigenvalue scales like
 $$
 |\lambda_\mu - \lambda_{\mu,N}| \leq C_\mu \frac{1}{\gamma_{\mu,N}} \varepsilon_N \varepsilon_N^*,
 $$
 where $C_\mu$ is a positive constant which only depends on the parameter $\mu$, where
 \begin{align*}
    &   \varepsilon_N := \underset{v_N \in \V}{\inf} \left\|u_{\mu}-v_N\right\|, \\
    &   \varepsilon_N^* := \underset{v_N \in \V}{\inf} \left\|u_{\mu}^*-v_N\right\|,
\end{align*}
and where $\gamma_{\mu,N}$ is the inf-sup constant of the reduced eigenvalue problem. Hence, it appears natural when it comes to the design of a greedy procedure in the present reduced basis context to enrich the Galerkin approximation space with snapshots of both the direct and adjoint eigenvalue problems, in order to at least get the reference solution for the reduced problem when considering the parameter $\mu$ of the snapshots.
\end{remark}

In the greedy procedure, the parameters $(\mu_n)_{\geq 1}$ need to be selected satisfying some criteria. In practice, it is common to choose a finite subset $\cP_{\rm train} \subset \cP$ of parameter values, called hereafter a {\itshape training set} and selecting the snapshots maximizing some error surrogate for the error between solutions of the reference model and the reduced model $\Delta_{N_n}$, as is described in Algorithm~\ref{algorithm:alg2}.
In an {\itshape ideal greedy} procedure, we would
take the exact error as the error surrogate
$\Delta_{N_n}$.
In that case, two possible choices for the definition of $\Delta_{N_n}$ would be:
\begin{itemize}
    \item [a)] either the error on the eigenvalue: $\Delta_{N_n}(\mu):= e^k_{N_n}(\mu)$
    \item [b)] or the error on the eigenvectors: $\Delta_{N_n}(\mu):= e^u_{N_n}(\mu) + e^{u^*}_{N_n}(\mu)$
\end{itemize}
with
$$
e^u_{N_n}(\mu) := \| u_\mu - u_{\mu, N_n}\|, \;  e^{u^*}_{N_n}(\mu) := \| u^*_\mu - u^*_{\mu, N_n}\| \; \mbox{ and } e^{k}_{N_n}(\mu) := | k_\mu - k_{\mu, N_n}|.
$$

However, these quantities are of course not available in general, so one has to resort to {\it a posteriori} error estimate for an {\itshape efficient greedy } algorithm. Therefore, the aim of the following section is to detail different strategies to define an {\it a posteriori} error estimator $\Delta_{N}(\mu)$ in order to obtain an estimation of the errors on the eigenvalues and the eigenvectors for any reduced space $\V$ without having to compute the solutions of the exact eigenvalue problem.

\begin{algorithm}
\caption{\textsc{Greedy Algorithm}}
\label{algorithm:alg2}
\begin{algorithmic}
     \STATE {\bf Input:} $\cP_{\text{train}} \subset \cP$: training set of parameters, $\tau>0$ : error tolerance threshold, $\V_0 \subset \R^{\mathcal N}$: initial reduced space
    \STATE{$N_0:= {\rm dim} \V_0$}
    \STATE{$\tau_0:=  \underset{\mu \in \cP_{\text{train}}}{\rm max } \; \Delta_{N_0}(\mu)$}
     \STATE{$n:=0$}
    \WHILE{$\tau_n > \tau$}
    \STATE{$\mu_{n+1}:= \underset{\mu \in \cP_{\text{train}}}{\rm argmax } \; \Delta_{N_n}(\mu)$ } 
    \STATE{Compute $u_{\mu_{n+1}}$ and $u^*_{\mu_{n+1}}$. }
   \STATE{$\V_{n+1}:= \V_n + {\rm Span}\{u_{\mu_{n+1}}, u^*_{\mu_{n+1}}\}$}
   \STATE{$N_{n+1}:= {\rm dim}\V_{n+1}$}
   \STATE{$\tau_{n+1}:= \underset{\mu \in \cP_{\text{train}}}{\rm max } \; \Delta_{N_{{n+1}}}(\mu)$}
   \STATE{$n:= n+1$}
    \ENDWHILE
    \STATE {\bf Output:} Reduced space $\V := \V_{n} \subset \R^{\mathcal N}$ 
\end{algorithmic}
\end{algorithm}

\section{{\it A posteriori} error estimation}
\label{sec:apost}

The goal of this section is to build {\it a posteriori} error bounds on the error between the solutions of the exact eigenvalue problems (\ref{eq:hfp_lambda}) and (\ref{eq:hfp_lambda_adjoint}), and the solutions of the reduced eigenvalue problems (\ref{eq:reducedpb}) and (\ref{eq:reducedpb_adjoint}). 

This section is organized as follows. Sections~\ref{sec:err_eigvec} and~\ref{sec:err_eigval} are respectively dedicated to error estimates on the eigenvectors and on the eigenvalue.  In Section~\ref{sec:theory}, we draw connections between the estimators we introduce in our context of non-symmetric eigenvalue problems, and the classical ones used for symmetric eigenvalue problems. Finally, we introduce in Section~\ref{sec:practical} the pratical {\em a posteriori} error estimators that will be used for numerical experiments in Section~\ref{sec:num}.

To simplify notation, the subscript $\mu$ is omitted in this section, as only one parameter value $\mu\in \cP$ is considered. Therefore, the quantities $u,u^*,\lambda,k$ are the solutions of the high fidelity problem, while $u_N,u_N^*,\lambda_N,k_N$ are the solutions of the reduced problem. We are therefore aiming at deriving bounds for the quantities
$e^k_N:=\left|k-k_N\right|$, $e^u_N:=\lVert u-u_N\rVert$, and $e^{u^*}_N:=\lVert u^*-u_N^*\rVert$. 
In order to estimate these errors, we first define the following residual vector quantities
\begin{align}
  &		 R_N = \left(B-k_N A\right)u_N, \label{eq:residual_R}\\
  &		 R_N^* = \left(B^T-k_N A^T\right)u_N^*. \label{eq:residual_Rstar}
\end{align}
We moreover define the vector 
\begin{equation}
	\tilde{u}^* = \dfrac{A^Tu^*}{\left\|A^Tu^*\right\|}, \label{eq:tilde_ustar}
\end{equation}
and the matrix
\begin{equation}
    \M = A^{-1} B,
\end{equation}
which is well-defined since $A$ is invertible from Assumption~\ref{assumption:Perron-Frobenius}. Note that it then holds that
\[
    M u = k u, \quad M^T \tilde{u}^* = k \tilde{u}^*.
\]

\subsection{Error estimates on the eigenvectors}\label{sec:err_eigvec}

Let $P \in \R^{\Nh\times\Nh}$ and $P^*\in \R^{\Nh\times\Nh}$ be the matrices associated with the spectral projection operators onto $\text{Span}\{\tilde{u}^*\}^{\perp}$ and $\text{Span}\{u\}^{\perp}$ respectively. More precisely, $P$ and $P^*$ are defined by
\begin{align}
&    P = I - \dfrac{u (\tilde{u}^*)^T}{\langle u, \tilde{u}^*\rangle},\label{eq:projector}\\
&    P^* = I - \dfrac{\tilde{u}^*  u^T}{\langle u, \tilde{u}^*\rangle},\label{eq:projector*}
\end{align}
where $I$ denotes the identity matrix of $\R^{\cN \times \cN}$. 
Before presenting the {\it a posteriori} error estimates, 
we first begin by collecting a few useful auxiliary lemmas. 

\begin{lemma}\label{lemma:def_spectral_projector}
    The spectral projector {onto the eigenspace of $M$ associated with the simple eigenvalue $k$} is $I-P$, where $P$ is defined by~\eqref{eq:projector}.
\end{lemma}
\begin{proof}
    Let us introduce the spectral projector $P_{\rm int}\in \R^{\cN \times \cN}$ of $M$ associated with the eigenvalue $k$, which is defined by: 
    \begin{equation*}
        \forall v \in \R^\cN, \quad P_{\rm int} v = \displaystyle \int_{{\mathcal C_{k}}} (z-M)^{-1}v \, dz, 
    \end{equation*}
    where $\mathcal C_k$ is a closed contour in the complex plane such that $k$ is the only eigenvalue of $M$ contained inside the loop. Let us show that $P_{\rm int}=\dfrac{u (\tilde{u}^*)^T}{\langle u, \tilde{u}^*\rangle}$.
    As the eigenvalue $k$ is simple, it holds that ${\rm Ran}\; P_{\rm int} = {\rm Span}\{u\}$, and $P_{\rm int}^T \tilde{u}^*=\tilde{u}^*$ by noting that $P_{\rm int}^T$ is the spectral projector associated with $M^T$ and the eigenvalue $k$. Let us show that ${\rm Ker}\; P_{\rm int} = ({\rm Span}\{{\tilde u}^*\})^{\perp}$. Indeed, for all $v\in \R^\cN$,
    \begin{align*}
        P_{\rm int}v = 0 \iff \langle \tilde{u}^*,v \rangle = \langle P_{\rm int}^T\tilde{u}^*,v \rangle = \langle \tilde{u}^*, P_{\rm int}v \rangle = 0 .
    \end{align*}
As ${\rm Ran}\; P_{\rm int} + {\rm Ker}\; P_{\rm int} = \mathbb{R}^{\mathcal{N}}$, we have ${\rm Span} \{u\} + [{\rm Span} \{\tilde u^*\}]^\perp = \R^\Nh$. The identity 
{$P_{\rm int} = I - P$ }
is then an immediate consequence of this decomposition. 
\end{proof}

\begin{lemma}\label{lemma:properties_proj}
There holds 
\begin{enumerate}
    \item[(i)] $P^2=P$;
    \item[(ii)] ${\rm Ker}\; P={\rm Span}\{u\}$, ${\rm Ran}\; P=[{\rm Span}\{\tilde{u}^*\}]^{\perp}$ and these two spaces are stable by $P$ and $M$;
    \item[(iii)] $\M P=P\M$.
\end{enumerate}
\end{lemma}
\begin{proof}
\begin{itemize}
\item[(i)] Let $v\in \R^{\Nh}$. Noting that $P u =0$, there holds
\[
        P^2 v  = P\left(v-\dfrac{\langle v,\tilde{u}^*\rangle}{\langle u, \tilde{u}^*\rangle}u\right) = P v - \dfrac{\langle v,\tilde{u}^*\rangle}{\langle u, \tilde{u}^*\rangle}P u = P v ,
\]
hence $P^2 = P$. 

\item[(ii)] The proof of the fact that ${\rm Ker}\; P={\rm Span}\{u\}$ and ${\rm Ran}\; P=[{\rm Span}\{\tilde{u}^*\}]^{\perp}$ is immediate from the proof of the previous lemma. The fact that ${\rm Ker}\; P$ is stable by $P$ and $M$ is also obvious. Now, let $v\in {\rm Span}\{\tilde{u}^*\}^{\perp}$, i.e. such that $\langle \tilde{u}^*, v\rangle = 0$. Then
\begin{align*}
    \langle \tilde{u}^*, P v\rangle & = 
    \langle \tilde{u}^*, v\rangle - \langle \tilde{u}^*, v\rangle \frac{\langle \tilde{u}^*, u\rangle}{\langle \tilde{u}^*, u\rangle}  = 0,
\end{align*}
and
\begin{align*}
    \langle \tilde{u}^*, M v\rangle & = 
     \langle M^T \tilde{u}^*, v\rangle = 
     k \langle \tilde{u}^*, v\rangle  = 0.
\end{align*}
Therefore $Pv \in {\rm Span} \{\tilde{u}^*\}^{\perp}$ and $\M v \in {\rm Span} \{\tilde{u}^*\}^{\perp}$.
\item[(iii)] It is obvious that for all $v\in {\rm Ker}\; P$, $PMv = MPv = 0$. Besides, for all $v\in {\rm Ran}\;P$, it holds that $Pv = v$, and $Mv \in {\rm Ran}\;P$ from (ii) so that $PMv = Mv$. As a consequence, $PMv = Mv = MPv$ for any $v\in \R^\Nh$, hence the desired result. 
\end{itemize}
\end{proof}

It is easy to check that the following Lemma holds on $P^*$, using similar arguments as in the proof of Lemma~\ref{lemma:properties_proj}.
\begin{lemma}\label{lemma:properties_proj_star}
There holds
\begin{itemize}
    \item[(i)] $(P^*)^2=P^*$;
    \item[(ii)] ${\rm Ker }\; P^* = {\rm Span}\{\tilde u^*\}$, ${\rm Ran}\; P^* = [{\rm Span} \{u\}]^\perp $ and these two spaces are stable by $P^*$ and $M^T$;  
    \item[(iii)] $M^T P^* = P^* M^T$.
\end{itemize}
\end{lemma}

{We have now gathered enough results to prove Lemma~\ref{lem:lem00}.}

\begin{proof}[Proof of Lemma~\ref{lem:lem00}] 
{Let us argue by contradiction. If $\langle u_{\mu}^*,A_{\mu} u_{\mu} \rangle = \langle\tilde{u}_{\mu}^*, u_{\mu} \rangle = 0$, then $\text{Span}\{u_{\mu}\} \subset (\text{Span}\{\tilde{u}_{\mu}^*\})^{\perp}$.
Yet, we have $\text{Span}\{u_{\mu}\} + (\text{Span}\{\tilde{u}_{\mu}^*\})^{\perp} = \mathbb{R}^{\mathcal{N}}$ using Lemma~\ref{lemma:properties_proj}-(\textit{ii}).
This yields to a contradiction and concludes the proof.}
\end{proof}

Let us introduce some notation. By Lemma~\ref{lemma:properties_proj}, the operator $PMP-k_NI$ leaves ${\rm Ran P} = [{\rm Span}\{\tilde{u}^*\}]^\perp$ invariant. 
Besides, provided that $k_N \notin \sigma\left(PMP|_{[{\rm Span\{\tilde{u}^*\}]^\perp}}\right)$, the operator $\left(PMP-k_NI\right)|_{[{\rm Span}\{\tilde{u}^*\}]^\perp}$ is invertible, 
seen as an operator from $[{\rm Span}\{\tilde{u}^*\}]^\perp$ onto $[{\rm Span}\{\tilde{u}^*\}]^\perp$. 
We can thus define the Moore--Penrose inverse of this operator, denoted by $\left(PMP-k_NI\right)^+$ and defined (by linearity) as follows: 
\begin{align*}
\forall v \in [{\rm Span}\{\tilde{u}^*\}]^\perp, & \left(PMP-k_NI\right)^+v = \left(PMP-k_NI\right)|_{[{\rm Span}\{\tilde{u}^*\}]^\perp}^{-1}v,\\
\forall v \in {\rm Span}\{u\}, & \left(PMP-k_NI\right)^+v=0.\\ 
\end{align*}
We define in a similar way the operator $\left(P^* M^T P^*-k_NI\right)^+$.

\begin{proposition}[Eigenvector estimation]
\label{prop:prop_proj}
Let $u_N,u_N^* \in \R^\Nh\setminus\{0\}$ and let $k_N\in \R$ such that $k_N\notin \sigma((PMP)|_{[{\rm Span}\{ \tilde{u}^*\}]^\perp})$ and $k_N\notin \sigma((P^*M^TP^*)|_{[{\rm Span}\{ u\}]^\perp}$.
Then, the following estimates hold:
 \begin{align}
         \inf_{v\in {\rm Span}\{u\}} \|u_N - v\|\leq  \; &
C_N^u
  \|R_N\|, \label{eq:res1}
\\
        \inf_{v^*\in {\rm Span}\{u^*\}} \|u_N^* - v^*\|\leq  \; &
C_N^{u^*} \|R_N^*\|, \label{eq:res2}
     \end{align}
     with 
     \begin{align*}
         C_N^u &:= \left\|P \left(P \M P -k_N I\right)^+ P A^{-1} \right\|,\\
         C_N^{u^*}& :=
\left\|A^{-T} P^* \left(P^* M^T P^*-k_NI\right)^+ P^*\right\|.\\     \end{align*}
Here and in the following, with a slight abuse of notation, we denote by $\|\cdot\|$ the operator norm associated with the vector norm $\|\cdot\|$ on $\R^{\mathcal N}$.
\end{proposition}

\begin{remark}
Note that $u_N, u_N^*$ and $k_N$ do not have to be respectively related to $u$, $u^*$ and $k$ for the above estimates to hold. However, in practice, $u_N$ will be an approximation of $u$, $u_N^*$ will be an approximation of $u^*$ and $k_N$ will be an approximation of $k$, so that the norms of the residuals $\|R_N\|$ and $\|R_N^*\|$ will be small.
\end{remark}

\begin{remark}  
The results obtained in Proposition~\ref{prop:prop_proj} match those of~\cite[Proposition 4]{giani2016robust} for $k=0$, noting the slightly different definition of the residual to take into account the generalized eigenvalue problem. 
\end{remark}

\begin{proof}
First, there holds
\begin{equation*}
	\inf_{v\in {\rm Span}\{u\}} \|u_N - v\| \leq \left\|P u_N\right\|.
\end{equation*}
Second, let us show that
$ P \left(P \M P -k_N I\right)^+ \left(P \M P -k_N I\right) P = P$. 
Indeed for $v \in {\rm Span} \{u\}$,\\
$P \left(P \M P -k_N I\right)^+ \left(P \M P -k_N I\right) P v = 0$ and $Pv = 0$. Besides, for $v \in [{\rm Span} \{\tilde u^*\}]^\perp, P v = v$, and
$\left(P \M P -k_N I\right) P v \in  [{\rm Span} \{\tilde u^*\}]^\perp$ from Lemma~\ref{lemma:properties_proj} (ii). As a consequence,
since $k_N\notin \sigma((PMP)_{|{[\rm Span}{\tilde u^*}]^\perp}$, $\left(P \M P -k_N I\right)$ is invertible on $[{\rm Span} \{\tilde u^*\}]^\perp$. 
Hence for $v\in [{\rm Span} \{\tilde u^*\}]^\perp,$ 
\[
P \left(P \M P -k_N I\right)^+ \left(P \M P -k_N I\right) P v = 
P \left(P \M P -k_N I\right)|_{{\rm Span}\{\tilde u^*\}}^{-1} \left(P \M P -k_N I\right)|_{{\rm Span}\{\tilde u^*\}} P v = Pv.
\]
We conclude by noting that $\R^\Nh = {\rm Span} \{u\} + [{\rm Span} \{\tilde u^*\}]^\perp$.

 Then using Lemma~\ref{lemma:properties_proj} (iii), we have
\begin{align*}
    P u_N &= 
    P \left(P \M P -k_N I\right)^+ \left(P \M P -k_N I\right) P u_N \\
    &= P \left(P \M P -k_N I\right)^+ 
    P\left(\M-k_N I\right)u_N.
\end{align*}
Using~\eqref{eq:residual_R}, we obtain
\begin{equation}
\label{eq:estimate_P_hu2}
    P u_N = P \left(P \M P -k_N I\right)^+ P A^{-1}R_N.
\end{equation}
Thus,
\begin{equation}
    \label{eq:estimate_P_hu}
	\left\|P u_N\right\| \leqslant \left\|P \left(P \M P -k_N I\right)^+ P A^{-1} \right\|  \|R_N\|.
\end{equation}
To show the second bound, we first note that
\begin{equation*}
	{P^*\left(A^Tu_N^*\right) = A^Tu_N^* - \dfrac{\langle u,A^Tu_N^* \rangle}{\langle u,\tilde{u}^* \rangle}\tilde{u}^*,}
\end{equation*}
so that using~\eqref{eq:tilde_ustar} and~\eqref{eq:projector*}
\begin{align*}
	\inf_{v^*\in {\rm Span}\{u^*\}} \|u_N^* - v^*\| &= \inf_{v^*\in {\rm Span}\{u^*
	\}} \left\|A^{-T}\left(A^Tu_N^*-A^Tv^*\right)\right\| \\
	&= \inf_{\tilde{v}^* \in {\rm Span}\{ \tilde{u}^*\}} \left\|A^{-T}\left(A^Tu_N^*- \tilde{v}^*\right)\right\| \\
	&\leqslant \left\|A^{-T}P^* A^Tu_N^*\right\|.
\end{align*}
Now, using Lemma~\ref{lemma:properties_proj_star} (iii) and similar arguments as above, there holds
\begin{align*}
     P^*\left(A^Tu_N^*\right) &= P^*
     \left(P^* M^T P^*-k_NI\right)^+
    \left(P^* M^T P^*-k_NI\right)
    P^* A^Tu_N^*
    \\
        & = P^* \left(P^* M^T P^*-k_NI\right)^+ P^*\left(M^T-k_NI\right)A^Tu_N^* \\
    &=  P^* \left(P^* M^T P^*-k_NI\right)^+ P^*\left(B-k_NA\right)^Tu_N^*.
\end{align*}
Hence
\begin{equation}
    \label{eq:estimate_Ph*}
    P^*\left(A^Tu_N^*\right) =  P^* \left(P^* M^T P^*-k_NI\right)^+ P^*R_N^*.
\end{equation}
Then,
\begin{equation*}
	\left\|A^{-T}P^* A^Tu_N^* \right\| \leqslant \left\|A^{-T} P^* \left(P^* M^T P^*-k_NI\right)^+ P^*\right\| \|R_N^*\|,
\end{equation*}
which proves~\eqref{eq:res2}.
\end{proof}

\subsection{Error estimate on the eigenvalue}\label{sec:err_eigval}

We now provide an estimate for the eigenvalue.

\begin{proposition}[Eigenvalue estimation]\label{prop:eigenvalue_aposteriori} Let $u_N, u_N^*\in \R^{\cN}$.
Under Assumption~\ref{assumption:Perron-Frobenius} and the assumptions that
$\displaystyle k_N := \frac{\langle u_N^*, Bu_N \rangle}{\langle u_N^*, A u_N \rangle}$ is such that $k_N\notin \sigma((PMP)|_{[{\rm Span}\{ \tilde{u}^*\}]^\perp})$ and $k_N\notin \sigma((P^*M^TP^*)|_{[{\rm Span}\{ u\}]^\perp}$,
there holds
\begin{equation}\label{eq:estimator}
    \left|k_N-k\right| \leq 
    C^k_N \eta_N^k,
\end{equation}
with 
\begin{equation}
    \label{eq:etaNK}
    \eta_N^k := \frac{\|R_N\| \|R_N^*\|}{\left|\langle u_N^*,A u_N\rangle\right|},
\end{equation}
and 
\begin{equation}
\label{eq:C_N}
        C^k_N := \left\| [P^* \left(P^* M^T P^*-k_NI\right)^+ P^*]^T (M - kI) P \left(P \M P -k_N I\right)^+ P A^{-1} \right\|.
\end{equation}
\end{proposition}

\begin{remark}
Note that in this result, the vectors $u_N$ and $u_N^*$ may not be solutions of a reduced eigenvalue problem of the form (\ref{eq:reducedpb}) or (\ref{eq:reducedpb_adjoint}). The only requirement of Proposition~\ref{prop:eigenvalue_aposteriori} is that $k_N$ has to be related to $u_N$ and $u_N^*$ by the formula stated in the proposition.
\end{remark}

\begin{proof}
For any $\alpha$, $\beta \in \mathbb{R}$,
\begin{align*}
    \langle A^T\left(u_N^*-\alpha u^*\right), \left(\M -k I\right)\left(u_N-\beta u\right)\rangle
    &= \langle A^T u_N^*,\M u_N \rangle 
    -\beta \langle A^Tu_N^*,\M u \rangle 
    -\alpha \langle A^Tu^*, \M u_N \rangle 
    +\alpha\beta \langle A^Tu^*,\M u \rangle \\
    & \quad - k \langle A^T u_N^*,u_N \rangle 
    +\beta k \langle A^Tu_N^*, u \rangle 
    + \alpha k \langle A^T u^*, u_N \rangle 
    -\alpha\beta k \langle A^Tu^*, u \rangle.
\end{align*}
Noting that $Mu = k u$, $M^T A^T u^* = k A^T u^*$ and recalling that $M=A^{-1}B$, we obtain
\begin{align*}
	 \langle A^T\left(u_N^*-\alpha u^*\right), \left(\M -kI\right)\left(u_N-\beta u\right)\rangle
	&= \langle A^T u_N^*, M u_N \rangle - k \langle u_N^*,Au_N \rangle\\
	&= \left(k_N-k\right)\langle u_N^*,Au_N \rangle.
\end{align*}
{According to Lemma~\ref{lem:lem00}, we can set}
\begin{equation*}
    \alpha = \dfrac{1}{\left\|A^Tu^*\right\|} \dfrac{\langle A^Tu_N^*,u\rangle}{\langle \tilde{u}^*,u\rangle},\quad \beta = \dfrac{\langle u_N,\tilde{u}^*\rangle}{\langle u,\tilde{u}^*\rangle},
\end{equation*}
so that we find 
\begin{equation*}
    k_N - k = \dfrac{1}{\langle u_N^*,Au_N\rangle}\langle P^*\left(A^Tu_N^*\right),\left(M-k I\right)P u_N\rangle .
\end{equation*}
Using~\eqref{eq:estimate_P_hu2} and~\eqref{eq:estimate_Ph*} finishes the proof.
\end{proof}

\section{Practical estimates of efficiencies and prefactors}\label{sec:prefactor}

As in the previous section, to simplify notation, the subscript $\mu$ is again omitted in this section.

In view of Proposition~\ref{prop:prop_proj} and Proposition~\ref{prop:eigenvalue_aposteriori}, it is natural to estimate the actual errors $e_N^k=\left|k-k_N\right|$, $e_N^u=\left\|u_{\mu,N}-u_{\mu}\right\|$ and $e_N^{u^*}=\left\|u^*_{\mu,N}-u^*_{\mu}\right\|$ by, respectively,
\begin{equation}\label{eq:Cbar}
\Delta_N^k := \overline{C}_N^k \eta_N^k,\qquad
\Delta_N^u := \overline{C}_N^u \|R_N\|,\qquad
\Delta_N^{u^*} := \overline{C}_N^{u^*} \|R^*_N\|,
\end{equation}
where $\overline{C}_N^k$, $\overline{C}_N^u$ and $\overline{C}_N^{u^*}$ are some constants which are 
good estimates of the efficiencies
	$\dfrac{e_N^k}{\eta_N^k(\mu)}$, 	$\dfrac{e_N^u}{\|R_N(\mu)\|}$, and $\dfrac{e_N^{u^*}}{\|R^*_N(\mu)\|}$. For example, one could use practical (computable) estimations of the constants $C_N^k$, $C_N^u$ and $C_N^{u^*}$ appearing in Proposition~\ref{prop:prop_proj} and Proposition~\ref{prop:eigenvalue_aposteriori}. 

\medskip

For the applications we are interested in, as will be illustrated below in  Section~\ref{sec:test_case_1}, we observe that the operators are perturbations of symmetric operators. This is why we investigate in Section~\ref{sec:theory} the  links between the prefactor $C_N^k$ introduced above and well-known results about this constant in the symmetric case. However, this does not yield practical efficient formulas for the prefactors. This is why  we propose in Section~\ref{sec:practical} a practical heuristic approach to compute some prefactors $\overline{C}_N^k$, $\overline{C}_N^u$ and $\overline{C}_N^{u^*}$ in the reduced basis context, that we use in the numerical results to build practical {\it a posteriori} error estimators in the greedy algorithm to select the reduced space. This heuristic approach gives very interesting numerical results for neutronics applications as will be illustrated in Section~\ref{sec:num}.

\subsection{Some connections with the symmetric case}\label{sec:theory}

The goal of this section is to draw links between the prefactor $C_N^k$ defined in Section~\ref{sec:err_eigval}, and the prefactor which is traditionnally used for the computation of {\it a posteriori} error estimators for symmetric eigenvalue problems. We begin this section by recalling well-known results about {\it a posteriori} error estimators for symmetric eigenvalue problems in Section~\ref{sec:sym}. 
In particular, we recall that the value of this prefactor in the symmetric context is directly linked with the value of the spectral gap of the exact problem. 
We then provide two different approaches to relate the value of the constant $C_N^k$ defined by (\ref{eq:C_N}) to the value of the prefactor in the symmetric case: (i) we prove in Section~\ref{sec:numrange} that in the non-symmetric case, the constant $C_N^k$ can be estimated using the distance between the approximate eigenvalue and the {\em numerical range} of the non-symmetric operator;
(ii) in Section~\ref{sec:perturbation} we study the perturbative regime where the non-symmetric operator can be seen as a small perturbation of a symmetric operator, and check that the prefactor for the non-symmetric operator is also a small perturbation of the well-known expression of the prefactor in the symmetric case.

\subsubsection{Symmetric case}\label{sec:sym}

The aim of this section is to recall some well-known results about {\it a posteriori} error estimators for symmetric eigenvalue problems, i.e. in the case where $A$ is a positive definite symmetric matrix and $B = I$, so that $M = A^{-1}$. In this case, all the eigenvalues of $M$ are real and positive, $k$ being its largest one. We still assume that $k$ is a simple eigenvalue of $M$ and denote by $k_2$ the second largest eigenvalue of $M$ so that $k > k_2$. Note that in the symmetric case, there holds that $u=  u^*$ and $P = P^* = P^T$.

As a consequence, for a given vector $u_N$ and the value $k_N = \dfrac{\langle u_N, Bu_N \rangle}{\langle u_N, A u_N \rangle} >0$, we have (from~\eqref{eq:C_N})
\begin{align*}
     C^k_N &= \left\| [P \left(P M P-k_NI\right)^+ P]^T (M - kI) P \left(P \M P -k_N I\right)^+ P A^{-1} \right\|
     \\
     & =  \left\| P \left(P M P-k_NI\right)^+ P (M - kI) P \left(P \M P -k_N I\right)^+ P A^{-1} \right\|,
\end{align*}
and we have the following proposition. 

\begin{proposition}\label{prop:prefactor_sym_case}
Let $A$ be symmetric positive definite and $B = I$. Let $k$ be the largest eigenvalue of $M=A^{-1}$, let us assume that it is simple, and let us denote by $k_2$ its second largest eigenvalue. Let us also assume that
\begin{equation}
\label{gap_assumption}
	k > k_N > k_{2} > 0.
\end{equation}
Then, there holds
\begin{equation}
    \label{eq:symmetric_case_prefactor}
    C^k_N = \frac{k_{2} (k-k_{2})}{(k_N-k_{2})^2} = \frac{\|PA^{-1}\|\|M- kI\|}{{\rm dist}(k_N, \sigma( (PMP)|_{{\rm Span}\{u\}^\perp}))^2}.
    \end{equation}
\end{proposition}

\begin{proof}
Denoting by $k=k_1 > k_2 \geq \cdots \geq k_\Nh$ the eigenvalues of $M = A^{-1}$ and by $u_1, u_2, \ldots, u_\Nh$ corresponding eigenvectors, there holds 
\[
    A^{-1} = M = \sum_{i=1}^\Nh k_i u_i u_i^T,
\]
and 
\[
    P = \sum_{i=2}^\Nh u_i u_i^T.
\]
Then, using functional calculus, there holds
\[
     P \left(P M P-k_NI\right)^+ P (M - kI) P \left(P \M P -k_N I\right)^+ P A^{-1} = 
     \sum_{i=2}^\Nh \frac{k_i(k-k_i)}{(k_N-k_i)^2} 
     u_i u_i^T.
\]
Since the operator norm is associated with the Euclidean vector norm, the operator norm corresponds to the largest eigenvalue of the (symmetric) operator, so that
\[
    C_N^k = \max_{2\leq i \leq \Nh} \frac{k_i(k-k_i)}{(k_N-k_i)^2} = \frac{k_2(k-k_2)}{(k_N-k_2)^2}.
\]
Since $\|PA^{-1}\| = k_2$, $\|M- kI\| = k-k_2$ using~\eqref{gap_assumption}, and ${\rm dist}(k_N, \sigma( (PMP)|_{{\rm Span}\{u\}^\perp}))^2 = (k_N-k_2)^2$, we easily obtain the second equality.
\end{proof}

The constant $C^k_N$ is therefore strongly linked to the spectral gap between the first and second eigenvalue of $M$ in this particular symmetric case. However, this notion of spectral gap is not clear in the non-symmetric context and we provide two points of view which enable to draw a comparison between the symmetric and non-symmetric context. 

\subsubsection{Numerical range}\label{sec:numrange}

In this section, we prove that in the general non-symmetric case, the value of the prefactor $C_N^k$ can be estimated using the so-called numerical range of the non-symmetric operator.
Let us first recall the notion of numerical range.

\begin{definition} Let $Q\in \R^{\cN \times \cN}$. 
The numerical range of the matrix $Q$ is  defined by
\begin{equation*}
    {\rm Num}(Q) = \overline{\left\{\langle v,Q v\rangle,\; \|v\|=1\right\}}.
\end{equation*}
\end{definition}

\begin{lemma}
\label{lem:num_range}
Let $Q\in \R^{\Nh\times \Nh}$ let and $z \notin \sigma(Q)$. Then,
\begin{equation*}
        \left\|\left(Q-z I\right)^{-1}\right\| 
        \leqslant 
        \dfrac{1}{\text{{\rm dist}}\left(z,{\rm Num}\left(Q\right)\right)}.
    \end{equation*}
\end{lemma}
\begin{proof}
Let $w\in\R^\Nh$ be a unit vector. There holds
\begin{align*}
        {\rm dist}\left(z,{\rm Num}\left(Q\right)\right) &\leqslant \left|z-\dfrac{\langle w,Q w\rangle}{\|w\|^2}\right|\\
        &\leqslant \dfrac{\left|\langle w, \left(Q-z I\right)w\rangle\right|}{\|w\|^2}\\
        &\leqslant \dfrac{\left\|\left(Q-z I\right)w\right\|}{\|w\|}.
    \end{align*}
    Taking $u=\left(Q-z I\right)w$, then,
    \begin{equation*}
        \left\|\left(Q-z I\right)^{-1}\right\| \leqslant \dfrac{1}{{\rm dist}\left(z,{\rm Num}\left(Q\right)\right)}.
    \end{equation*}
\end{proof}

\begin{proposition}\label{prop:num_range_upper_bound} Under the same assumptions as in Proposition~\ref{prop:eigenvalue_aposteriori}, there holds
\[
    C^k_N \le   \frac{\| M-k I \| \|PA^{-1} \| }{{\rm dist}(k_N, {\rm Num} ((PMP)_{|[{\rm Span} \{\tilde u^*\}]^\perp}) ) \; {\rm dist}(k_N, {\rm Num} (P^* M^T P^*)_{|[{\rm Span} \{ u\}]^\perp}))}
\]
\end{proposition}

Note that the bound given by Proposition~\ref{prop:num_range_upper_bound} is exactly equal to the value of the prefactor in the symmetric case since when $M$ is symmetric and non-negative, ${\rm Num} ((PMP)_{|[{\rm Span} \{\tilde u^*\}]^\perp})={\rm Num} (P^* M^T P^*)_{|[{\rm Span} \{ u\}]^\perp}=[k_2, k_\Nh]$ where $k > k_2 \geq \ldots \geq k_\Nh$ are the ordered eigenvalues of $M$.  

\begin{proof}
Starting from~\eqref{eq:C_N}, it holds that
\[
    C^k_N \le \left\| P^* \left(P^* M^T P^*-k_NI\right)^+ P^* \right\|
    \| M - kI \|
    \left\| P \left(P \M P -k_N I\right)^+ P\right\| \|P A^{-1} \| .
\]
Using Lemma~\ref{lem:num_range}, we easily obtain the result. 
\end{proof}

The upper bound on $C^k_N$ stated in Proposition~\ref{prop:num_range_upper_bound} goes to infinity if $k_N$ (which is supposed to be an approximation of $k$) gets close to ${\rm Num} ((PMP)_{|[{\rm Span} \{\tilde u^*\}]^\perp})$ or ${\rm Num} (P^* M^T P^*)_{|[{\rm Span} \{ u\}]^\perp}$, which can be seen as an underlying spectral gap assumption.

\subsubsection{A Perturbative approach}\label{sec:perturbation}

The aim of this section is to propose another connection between the estimation of the prefactor $C_N^k$ in the non-symmetric case with its well-known expression in the symmetric case. In all this section, we assume that
\begin{equation}\label{eq:perturb_hyp}
\left\{
    \begin{aligned}
&    A = A^\varepsilon= S + \varepsilon T\text{ with } S^T=S,  \, T^T=-T, \, \varepsilon>0,\\
& B=I.
\end{aligned}
\right.
\end{equation}
In other words, the matrix $A$ is a perturbation of a symmetric positive definite matrix $S\in \R^{\Nh \times \Nh}$, since $\varepsilon>0$ is intended to be a small parameter. We still assume here that $B = I$ for the sake of simplicity. 

We also assume that the positive definite symmetric matrix $S$ has a simple positive lowest eigenvalue $\lambda_S$, and that $u_S$ is an associated eigenvector. 
We then denote by $\lambda_S < \lambda_{S,2} \leq \ldots \leq \lambda_{S, \Nh}$ all the eigenvalues of~$S$. We also denote by $k_S:= \frac{1}{\lambda_S}$ and by $k_{S,i}:= \frac{1}{\lambda_{S,i}}$ for $2\leq i \leq \Nh$.
By a perturbative argument, for any $\varepsilon>0$ small enough, 
there exists a simple nonzero eigenvalue $\lambda^\varepsilon$ of smallest modulus of $A^\varepsilon$, and we denote by $u^\varepsilon$ 
an associated direct eigenvector, $u^{*,\varepsilon}$ an associated adjoint eigenvector, $\tilde u^{*,\varepsilon}$ defined as in~\eqref{eq:tilde_ustar} and 
$k^\varepsilon:= \frac{1}{\lambda^\varepsilon}$. 

\medskip

For the sake of simplicity of the perturbative analysis, we assume that the approximate value $k_N$ is independent of~$\varepsilon$. This for example makes sense if one uses a reduced-order model constructed from the one-dimensional reduced space $\mathcal V = {\rm Span}\{u_S\}$. In that case, $u_N = u_N^* = u_S$ and thus $k_N = k_S$.

In this section, using obvious notation, we would like to study the convergence of the prefactor 
$$
C_N^{k,\varepsilon}:= \left\| [P^{*,\varepsilon} \left(P^{*,\varepsilon} (M^\varepsilon)^T P^{*,\varepsilon}-k_NI\right)^+ P^{*,\varepsilon}]^T (M^\varepsilon - k^\varepsilon I) P^\varepsilon \left(P^\varepsilon M^\varepsilon P^\varepsilon -k_N I\right)^+ P^\varepsilon (A^\varepsilon)^{-1} \right\|
$$
to the value
\begin{equation}
    \label{eq:CNksym}
    C_N^{k,\rm sym} = \frac{k_{S,2} (k_{S}-k_{S,2})}{(k_N-k_{S,2})^2}
\end{equation}
as $\varepsilon$ goes to $0$.

We first perform a first-order expansion of the eigenvectors and eigenvalues in $\varepsilon$ (cf. Chapter 2 of ~\cite{kato2013perturbation}).

\begin{lemma}
\label{prop:perturbation_expansion}
Let us assume~\eqref{eq:perturb_hyp} and
\begin{equation}
	\|u^\varepsilon\|^2 = \|u_S\|^2 = 1 \quad \mbox{ and } \quad \langle u^\varepsilon, u_S \rangle>0. \label{eq:assumption_norm_perturbation}
\end{equation}
Then, as $\varepsilon$ goes to $0$,
\begin{align*}
	&	\lambda^\varepsilon = \lambda_S + O(\varepsilon^2), \\
	& k^\varepsilon = k_S + O(\varepsilon^2),  \\
	&	u^\varepsilon = u_S - \varepsilon\left(S-\lambda_S I\right)^{-1}_{|{\rm Span}\{u_S\}^{\perp}}T u_S + O(\varepsilon^2), \nonumber \\
	&	u^{*,\varepsilon} = u_S + \varepsilon\left(S -\lambda_S I\right)^{-1}_{|{\rm Span}\{u_S\}^{\perp}}T u_S + O(\varepsilon^2), \label{eq:solution_HF_perturbed} \\
  & \tilde u^{*,\varepsilon} = u_S + \varepsilon\left(S -\lambda_S I\right)^{-1}_{|{\rm Span}\{u_S\}^{\perp}}T u_S
    + O(\varepsilon^2).
\end{align*}
\end{lemma}

\begin{proof}
Using the results of~\cite[Chapter~2]{kato2013perturbation}, we decompose $\lambda^\varepsilon$, $u^\varepsilon$, and $u^{*,\varepsilon}$ at first order as 
\begin{align*}
	& \lambda^\varepsilon = 
	\lambda_{A,0}+\varepsilon \lambda_{A,1} + O(\varepsilon^2),
	\\
	& u^\varepsilon =  u_{A,0}+\varepsilon u_{A,1} + O(\varepsilon^2),
	\\
	& u^{*,\varepsilon} =  u^*_{A,0}+\varepsilon u^*_{A,1} + O(\varepsilon^2).
\end{align*}
Using this decomposition, the eigenvalue problem writes
\[
 S u_{A,0} + \varepsilon (S u_{A,1} + T u_{A,0})
 = \lambda_{A,0} u_{A,0} + \varepsilon (\lambda_{A,0} u_{A,1} +
 \lambda_{A,1} u_{A,0} ) +  O(\varepsilon^2).
\]
At order 0 in $\varepsilon$, we obtain $u_{A,0} = u_S$ and $\lambda_{A,0} = \lambda_S$. Then, at first order,
\begin{equation}
	S u_{A,1} + T u_{A,0} = \lambda_{A,0} u_{A,1} + \lambda_{A,1} u_{A,0} \label{eq:perturbation_order_1}.
\end{equation}
Using \eqref{eq:assumption_norm_perturbation}, one can write
\begin{equation*}
	\|u^\varepsilon\|^2 = \|u_{A,0}\|^2 + \varepsilon (\langle u_{A,0}, u_{A,1} \rangle + \langle u_{A,1}, u_{A,0} \rangle) + O(\varepsilon^2),
\end{equation*}
which implies that
\begin{equation}
	\langle u_{A,0},u_{A,1}\rangle = 0 .
\end{equation}
Using the latter and projecting \eqref{eq:perturbation_order_1} onto $u_{A,0}$ gives
\begin{equation*}
	\langle S u_{A,1},u_{A,0}\rangle + \langle T u_{A,0},u_{A,0}\rangle = \lambda_{A,1} \langle u_{A,0},u_{A,0}\rangle = \lambda_{A,1}.
\end{equation*}
As $T$ is skew-symmetric, it holds $\langle T u_{A,0},u_{A,0}\rangle = 0$, so that
\begin{equation*}
	\langle S u_{A,1},u_{A,0}\rangle = \langle u_{A,1},S u_{A,0}\rangle = \lambda_{A,0} \langle u_{A,1},u_{A,0}\rangle = 0.
\end{equation*}
Hence $\lambda_{A,1}=0.$
Then~\eqref{eq:perturbation_order_1} transforms into
\begin{equation*}
	\left(S-\lambda_S I\right)u_{A,1} = -T u_S.
\end{equation*}
The latter has a solution since $ T u_S \in {\rm Span}\{u_S\}^{\perp}$ and $\left(\text{Ker}\left(S-\lambda_S I\right)\right)^{\perp} = \text{Ran}\left(S-\lambda_S I\right).$
Hence 
\begin{equation}
    \label{eq:uA1}
     u_{A,1} = -  \left(S-\lambda_S I\right)^{-1}_{|{\rm Span}\{u_S\}^\perp} T u_S.
\end{equation}
We can apply the same procedure for the adjoint eigenvector to obtain the result. Finally, 
\begin{align*}
    \tilde u^{*,\varepsilon} &:= 
    \frac{(A^\varepsilon)^T u^{*,\varepsilon}}{\|(A^\varepsilon)^T u^{*,\varepsilon}\|}
    \\
    & = 
    \frac{(S-\varepsilon T) (u_S-\varepsilon u_{A,1})}{\| (S-\varepsilon T) (u_S-\varepsilon u_{A,1})\|} 
    + O(\varepsilon^2)
    \\
    & = 
    \frac{\lambda_S u_S - \varepsilon (Su_{A,1} + Tu_S)}{\| \lambda_S u_S - \varepsilon (Su_{A,1} + Tu_S)\|}+ O(\varepsilon^2)
    \\
    & = 
    \left(u_S - \frac{\varepsilon}{\lambda_S} (Su_{A,1} + Tu_S)\right)
    \left(1 + \frac{\varepsilon}{\lambda_S} \langle u_S,(Su_{A,1} + Tu_S)\rangle \right)
    + O(\varepsilon^2)
    \\
    & = 
    u_S - \varepsilon  u_{A,1}
    + O(\varepsilon^2),
\end{align*}
which concludes the proof.
\end{proof}

We now provide first-order expansions of operators which will be needed in the subsequent estimation of the prefactor.

\begin{lemma}
Let us assume~\eqref{eq:perturb_hyp} and~\eqref{eq:assumption_norm_perturbation}.
Then, as $\varepsilon$ goes to $0$,
$$
        P^\varepsilon = P_S + \varepsilon P_T + O(\varepsilon^2) \text{ and }
        P^{*,\varepsilon} = P_S - \varepsilon P_T + O(\varepsilon^2),
$$
    where
$$
       P_S = I - u_S u_S^T, \text{ and }
       P_T = u_{A,1} u_S^T - u_S u_{A,1}^T,
$$
    $u_{A,1} $ being defined in~\eqref{eq:uA1}.
\end{lemma}

\begin{proof}
    We have
\begin{align*}
      P^\varepsilon & = I - \dfrac{u^\varepsilon (\tilde{u}^{*,\varepsilon})^T}{\langle u^\varepsilon, \tilde{u}^{*,\varepsilon}\rangle},\\
      & = I - (u_S + \varepsilon u_{A,1}+ O(\varepsilon^2)) (u_S - \varepsilon u_{A,1}+ O(\varepsilon^2))^T \\
      & = I - u_S u_S^T +\varepsilon (u_{A,1} u_S^T - u_S u_{A,1}^T) + O(\varepsilon^2).
\end{align*}
Similarly, 
\begin{align*}
   P^{*,\varepsilon} &= I - \dfrac{\tilde{u}^{*,\varepsilon}  (u^\varepsilon)^T}{\langle u^\varepsilon, \tilde{u}^{*,\varepsilon}\rangle} \\
   & = I - (u_S - \varepsilon u_{A,1}+ O(\varepsilon^2)) (u_S + \varepsilon u_{A,1}+ O(\varepsilon^2))^T \\
      & = I - u_S u_S^T +\varepsilon (- u_{A,1} u_S^T + u_S u_{A,1}^T)+ O(\varepsilon^2).
\end{align*}
This concludes the proof.
\end{proof}

We now provide a first order expansion of the operator entering the prefactor in~\eqref{eq:C_N}, namely
\begin{equation}
    \mathcal{M}^\varepsilon =  [P^\varepsilon \left(P^\varepsilon M^\varepsilon P^\varepsilon-k_NI\right)^+ P^\varepsilon] (M^\varepsilon - k^\varepsilon I) P^\varepsilon \left(P^\varepsilon M^\varepsilon P^\varepsilon -k_N I\right)^+ P^\varepsilon (A^\varepsilon)^{-1}.
\end{equation}

\begin{lemma}
Let us assume~\eqref{eq:perturb_hyp} and~\eqref{eq:assumption_norm_perturbation}. 
Then, as $\varepsilon$ goes to $0$,
    \begin{equation}
    \label{eq:3.30}
        \mathcal{M}^\varepsilon = \mathcal{M}_0 + \varepsilon \mathcal{M}_1  + O(\varepsilon^2),
    \end{equation}
    with 
    \begin{align*}
   \mathcal{M}_0 &=\Gamma_S(S^{-1} - k_{S}I)\Gamma_S S^{-1},\\
   \mathcal{M}_1 &=-\Gamma_SS^{-1}TS^{-1}\Gamma_SS^{-1} + \Gamma_S(S^{-1} - k_{S}I)\Gamma_TS^{-1}
   -\Gamma_S(S^{-1} - k_{S}I)\Gamma_S S^{-1}TS^{-1} + \Gamma_T(S^{-1} - k_{S}I)\Gamma_S S^{-1},
\end{align*}
and 
\begin{align*}
    \Gamma_S & = P_S\left(P_SS^{-1}P_S -k_N I\right)^+P_S  ,\\
       \Gamma_T  &= P_S\left(P_SS^{-1}P_S -k_N I\right)^+P_T + P_T\left(P_SS^{-1}P_S -k_N I\right)^+P_S \\
   & \quad   - P_S\left(P_SS^{-1}P_S -k_N I\right)^+\left(P_SS^{-1}TS^{-1}P_S + P_TS^{-1}P_S + P_SS^{-1}P_T\right)\left(P_SS^{-1}P_S -k_N I\right)^+P_S  .
\end{align*}
\end{lemma}

\begin{proof}
First, 
\[
    M^\varepsilon = (A^\varepsilon)^{-1}=  S^{-1} - \varepsilon S^{-1} T S^{-1}+ O(\varepsilon^2).
\]
Therefore, as we have $P^\varepsilon = P_S + \varepsilon P_T +O(\varepsilon^2)$, there holds
\begin{align*}
    P^\varepsilon M^\varepsilon P^\varepsilon &= \left(P_S + \varepsilon P_T +O(\varepsilon^2)\right)\left(S^{-1} - \varepsilon S^{-1} T S^{-1}+ O(\varepsilon^2)\right)\left(P_S + \varepsilon P_T + O(\varepsilon^2)\right) \\
    & = \left(P_SS^{-1} - \varepsilon P_SS^{-1}TS^{-1} + \varepsilon P_TS^{-1}+ O(\varepsilon^2)\right)\left(P_S + \varepsilon P_T + O(\varepsilon^2) \right) \\
    & = P_SS^{-1}P_S - \varepsilon P_SS^{-1}TS^{-1}P_S + \varepsilon P_TS^{-1}P_S + \varepsilon P_SS^{-1}P_T + O(\varepsilon^2)\\
    & = P_SS^{-1}P_S + \varepsilon \left(P_SS^{-1}TS^{-1}P_S + P_TS^{-1}P_S + P_SS^{-1}P_T\right)+ O(\varepsilon^2) .
\end{align*}
Using a first-order expansion of the pseudo-inverse in $\varepsilon$, there holds
\begin{align*}
  (P^\varepsilon &M^\varepsilon P^\varepsilon -k_N I)^+ = \left[\left(P_SS^{-1}P_S -k_N I\right) + \varepsilon \left(P_SS^{-1}TS^{-1}P_S + P_TS^{-1}P_S + P_SS^{-1}P_T\right)+ O(\varepsilon^2)\right]^+ \\
    =& \left(P_SS^{-1}P_S -k_N I\right)^+  \\
    &- \varepsilon\left(P_SS^{-1}P_S -k_N I\right)^+\Big(P_SS^{-1}TS^{-1}P_S + P_TS^{-1}P_S     + P_SS^{-1}P_T\Big)\left(P_SS^{-1}P_S -k_NI\right)^+ + O(\varepsilon^2).
\end{align*}
Hence, one can write
\begin{align*}
   & P^\varepsilon\left(P^\varepsilon \M^\varepsilon P^\varepsilon -k_N I\right)^+P^\varepsilon 
    = P_S\left(P_SS^{-1}P_S -k_N I\right)^+P_S\\ 
&     + \varepsilon \Bigg[P_S\left(P_SS^{-1}P_S -k_N I\right)^+P_T  
     - P_S\left(P_SS^{-1}P_S -k_N I\right)^+\Big(P_SS^{-1}TS^{-1}P_S 
    + P_TS^{-1}P_S + P_SS^{-1}P_T\Big)\\
     &\qquad \quad \times \left(P_SS^{-1}P_S -k_N I\right)^+P_S  + P_T\left(P_SS^{-1}P_S -k_N I\right)^+P_S\Bigg]+ O(\varepsilon^2).
\end{align*}
Defining
\begin{align*}
    \Gamma^\varepsilon := P^\varepsilon\left(P^\varepsilon \M^\varepsilon P^\varepsilon -k_N I\right)^+P^\varepsilon, 
\end{align*}
we have just obtained that
\begin{align*}
    \Gamma^\varepsilon  = \Gamma_S + \varepsilon\Gamma_T+ O(\varepsilon^2).
\end{align*}
Using that
\begin{align*}
   \mathcal M^\varepsilon =  \Gamma^\varepsilon (M^\varepsilon - k^\varepsilon I)\Gamma^\varepsilon (A^\varepsilon)^{-1} &= \left(\Gamma_S + \varepsilon\Gamma_T+ O(\varepsilon^2)\right)\left(S^{-1} - k_{S}I - \varepsilon S^{-1} T S^{-1}+ O(\varepsilon^2)\right)\\
    &\quad \times \left(\Gamma_S + \varepsilon\Gamma_T+ O(\varepsilon^2)\right)\left(S^{-1} - \varepsilon S^{-1} T S^{-1}+ O(\varepsilon^2)\right),
\end{align*}
we easily obtain~\eqref{eq:3.30}.
\end{proof}

We then estimate the prefactor $C_N^k$ in the perturbative case using the previous results.

\begin{proposition}\label{prop:prefactor_upper_bound}
Let us assume~\eqref{eq:perturb_hyp} and~\eqref{eq:assumption_norm_perturbation}. Let us also assume that
\begin{equation}
\label{gap_assumption2}
	k_{S} \geq k_N > k_{S,2} > 0,
\end{equation}
and that $k_{S,2}$ is not degenerate. 
Then, for $\varepsilon$ sufficiently small, there holds
\begin{align*}
    &C_N^{k, \varepsilon} = C_N^{ k,\rm sym} + O(\varepsilon^2),
C_N^{k,\rm skew} + O(\varepsilon^2).
\end{align*}
where $C_N^{ k,\rm sym}$ is defined by~\eqref{eq:CNksym}.
\end{proposition}

\begin{proof}

Starting from~\eqref{eq:3.30}, let us first note that $\mathcal{M}_0 = \Gamma_S(S^{-1} - k_{S}I)\Gamma_S S^{-1}$
has the same spectral decomposition as $S$, that is eigenvectors $u_{S,i}$ with corresponding eigenvalues
\[
     \left\{
    \begin{array}{ll}
        0 & \mbox{ for } i = 1 \\
        \frac{(k_{S,i}-k_S)k_{S,i}}{(k_{S,i}- k_N)^2} & \mbox{ for } 2 \le i \le \Nh.
    \end{array}
\right.
\]
From this, we deduce that
\begin{align*}
\|\mathcal{M}_0\|  = \max_{2\leq i \leq \Nh} \frac{|k_{S,i}-k_S|k_{S,i}}{|k_{S,i}- k_N|^2}
 = \frac{|k_{S,2}-k_S|k_{S,2}}{|k_{S,2}- k_N|^2} = C_N^{k, {\rm sym}}.
\end{align*}
Note also that the same holds for $\Gamma_S$ with eigenvalues
\[
     \left\{
    \begin{array}{ll}
        0 & \mbox{ for } i = 1 \\
        \frac{1}{k_{S,i}- k_N} & \mbox{ for } 2 \le i \le \Nh.
    \end{array}
\right.
\]

Then, noting that $k_{S,2}$ is a simple eigenvalue, we can
write down the Taylor expansion of the spectral norm as
\begin{align*}
    C^{k, \varepsilon}_N  = \|\mathcal{M}_0 + \varepsilon \mathcal{M}_1 + O(\varepsilon^2) \| =   \|\mathcal{M}_0\|
    + \varepsilon  u_{\mathcal{M},0}^T \mathcal{M}_1 u_{\mathcal{M},0} + O(\varepsilon^2),
\end{align*}
where $u_{\mathcal{M},0}$ the unit eigenvector corresponding to the largest eigenvalue of $\mathcal{M}_0$, 
that is $u_{\mathcal{M},0} = \pm u_{S,2}$. For simplicity, let us choose
$u_{\mathcal{M},0} = u_{S,2}$.

Then
\begin{align*}
 u_{\mathcal{M},0}^T \mathcal{M}_1 u_{\mathcal{M},0} &= -u_{\mathcal{M},0}^T\Gamma_SS^{-1}TS^{-1}\Gamma_SS^{-1}u_{\mathcal{M},0}
 + u_{\mathcal{M},0}^T\Gamma_S(S^{-1} - k_{S}I)\Gamma_TS^{-1}u_{\mathcal{M},0} \\
    &\quad -u_{\mathcal{M},0}^T\Gamma_S(S^{-1} - k_{S}I)\Gamma_S S^{-1}TS^{-1}u_{\mathcal{M},0} 
    + u_{\mathcal{M},0}^T\Gamma_T(S^{-1} - k_{S}I)\Gamma_S S^{-1}u_{\mathcal{M},0}\\
     &= 
     - \quad \frac{k_{S,2}^3}{(k_{S,2}-k_N)^2}
     u^T_{\mathcal{M},0} Tu_{\mathcal{M},0}
     +\frac{k_{S,2}(k_{S,2}-k_S)}{k_{S,2}-k_N}u^T_{\mathcal{M},0}\Gamma_Tu_{\mathcal{M},0} \\
    &\quad -\frac{k_{S,2}^2(k_{S,2}-k_S)}{(k_{S,2}-k_N)^2}u^T_{\mathcal{M},0} Tu_{\mathcal{M},0} 
    +\frac{k_{S,2}(k_{S,2}-k_S)}{k_{S,2}-k_N} u_{\mathcal{M},0}^T\Gamma_T u_{\mathcal{M},0}\\
    &=0,
\end{align*}
where we used that the matrices $T$ and $\Gamma_T$ are skew-symmetric.
This concludes the proof.
\end{proof}

We now illustrate the above bounds on toy numerical examples. Let us introduce the following matrices $S,T\in \R^{4\times 4}$
:
\begin{equation*}
    S = \begin{pmatrix} 2000 & 0 & 0 & 0 \\
    0 & 1500 & 0 & 0 \\
    0 & 0 & 1000 & 0 \\
    0 & 0 & 0 & 0.02 \end{pmatrix},\quad 
    T_0 = \begin{pmatrix} 0 & 1 & 1 & 1 \\
    -1 & 0 & 1 & 1 \\
    -1 & -1 & 0 & 1 \\
    -1 & -1 & -1 & 0 \end{pmatrix},\quad T = \frac{\|S\|}{\|T_0\|}T_0 .
\end{equation*}
It then holds that $k_S = 50$ and $ k_{S,2} = 0.001$.  Let us consider $k_N=k_S$. A second-order convergence of the difference $|C_N^{k,\varepsilon} - C_N^{k,\rm sym}|$ as a function of $\varepsilon$ is observed on Figure~\ref{fig:perturbation_2}. This is a strong indication that the estimate of Proposition~\ref{prop:prefactor_upper_bound} is sharp.

In our practical applications of interest, we will indeed observe that the operator is a perturbation of a symmetric operator, but the estimate of the prefactor by the one obtained using the symmetric part is not sufficiently good over a large range of the values of the parameters $\mu$, in particular because the spectral gap (see Assumption~\eqref{gap_assumption2}) is not uniformly bounded from below (see Section~\ref{sec:test_case_1} for a discussion). This is why we will resort to  a practical heuristic method to approximate the prefactor, as is now explained in the next section~\ref{sec:practical}.

\begin{figure}[h!]
    \centering
    \includegraphics[scale=0.6]{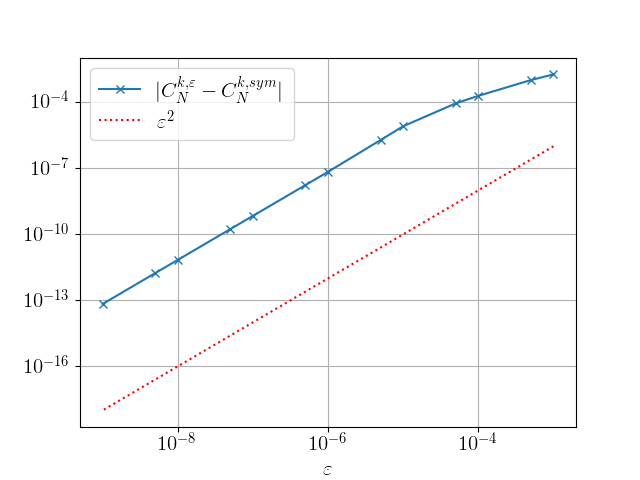}
    \caption{$|C_N^{k,\varepsilon} - C_N^{k,sym}|$ as a function of $\varepsilon$}
    \label{fig:perturbation_2}
\end{figure}


\subsection{Practical {\it a posteriori} error estimator}\label{sec:practical}

The aim of this section is to present the heuristic algorithm we use in order to estimate the prefactors $C_N^k$, $C_N^u$ and $C_N^{u^*}$ defined in Proposition~\ref{prop:eigenvalue_aposteriori} and Proposition~\ref{prop:prop_proj} respectively. The algorithm then yields approximations of these constants, denoted by 
$\overline{C}_N^k$, $\overline{C}_N^u$ and $\overline{C}_N^{u^*}$, which are used to build {\it a posteriori} error estimates for the greedy algorithm presented in Algorithm~\ref{algorithm:alg2}.

This heuristic procedure is based on the use of an estimation set of parameters values $\cP_{\rm pref} \subset \cP$, containing a finite number of elements, which does not contain any values of the parameters belonging to the training set $\cP_{\rm train}$. In other words, the estimation set $\cP_{\rm pref}$ is chosen so that $\cP_{\rm train} \cap \cP_{\rm pref} = \emptyset$. High-fidelity solutions of the eigenvalue problems (\ref{eq:hfp_lambda}) are computed in the offline phase for all $\mu \in \cP_{\rm pref}$.

For a given reduced space $V_N$, let us introduce the efficiency ratios: for all $\mu \in \cP$, 
\begin{equation}
\label{eq:ENk}
	\mathcal{E}^k_N(\mu) := \dfrac{\left|k_{\mu,N}-k_{\mu}\right|}{\eta_N^k(\mu)}, \; 	\mathcal{E}^u_N(\mu) := \dfrac{\left\|u_{\mu,N}-u_{\mu}\right\|}{\|R_N(\mu)\|} \; \mbox{ and } \; \mathcal{E}^{u^*}_N(\mu) := \dfrac{\left\|u^*_{\mu,N}-u^*_{\mu}\right\|}{\|R^*_N(\mu)\|} .
\end{equation}
The latter quantities are computed in the offline phase for all $\mu \in \cP_{\rm pref}$. 
By definition, it holds that for all $\mu \in \mathcal{P}$,
\begin{equation}
    \mathcal{E}^k_N(\mu) \le C_N^k(\mu), \;   \mathcal{E}^u_N(\mu) \le C_N^u(\mu) \; \mbox{ and } \; \mathcal{E}^{u^*}_N(\mu) \le C_N^{u^*}(\mu) .
\end{equation}

Our heuristic approach then consists in estimating the constants $C_N^k(\mu)$, $C_N^u(\mu)$ and $C_N^{u^*}(\mu)$ for all $\mu\in \cP$ by their maximum values over $\cP_{\rm pref}$. More precisely, defining
\begin{equation}\label{eq:Cbarmax}
 \overline{C}_N^k  := \max_{\mu \in \cP_{\rm pref}} \mathcal{E}^k_N(\mu), \qquad
        \overline{C}_N^u  := \max_{\mu \in \cP_{\rm pref}}\mathcal{E}^u_N(\mu), \text{ and } \qquad
     \overline{C}_N^{u^*}  := \max_{\mu \in \cP_{\rm pref}} \mathcal{E}^{u^*}_N(\mu),    
\end{equation}
the practical {\it a posteriori} error estimates used in the greedy algorithm are then defined by

\begin{equation}\label{eq:Cbarmax_delta}
\Delta_N^k(\mu) := \overline{C}_N^k \eta_N^k(\mu),\qquad
\Delta_N^u(\mu) := \overline{C}_N^u \|R_N(\mu)\|,\text{ and }\qquad
\Delta_N^{u^*}(\mu) := \overline{C}_N^{u^*} \|R^*_N(\mu)\|.
\end{equation}

The efficiency of this practical approach will be illustrated in the next section, where numerical results obtained in neutronics applications are presented. 

\section{Numerical results}\label{sec:num}

The aim of this section is to illustrate the behaviour of the proposed reduced basis method on examples arising from neutronics applications. The considered physical model is presented in Section~\ref{sec:physical_context}. 
In Section~\ref{sec:high-fidelity_discretization}, we describe the high-fidelity discretization of the problem using a finite element method. 
The parametric dependency of the coefficients of the mathematical equations describing the model enables the matrices to be assembled using a so-called affine decomposition, discussed in Section~\ref{sec:affine_decomposition}. 
The eigenvalue solver is described in Section~\ref{sec:eigenvalue_solver}. Finally, numerical tests presented in Section~\ref{sec:numerical_tests} give an application of the reduced basis method to nuclear core computations.

\subsection{ The continuous model: two-group neutron diffusion equations}
\label{sec:physical_context}
The stationary neutron flux density in a reactor core is determined by solving the transport equation which depends on six variables: position (3-dimensional), velocity direction (2-dimensional), the velocity norm or energy (1-dimensional). It physically states the balance between the emission of neutrons by fission and the absorption, scattering, and leakage of neutrons at the boundary of the spatial domain. 
The most common discretization of the energy variable is the multigroup approximation where the energy domain is divided into subintervals called energy groups. 
The reactor core is modeled by a bounded, connected and open subset $\Omega$ of $\mathbb{R}^d$ (where typically $d=3$, but one ot two-dimensional are also considered) having a Lipschitz boundary which is piecewise regular. In practice, the neutron flux density is usually modeled by the multigroup neutron diffusion equations~\cite[Chapter 7]{DuHa76} at the reactor core scale.

Let us now make precise the specific two-group neutron diffusion model we consider in this work. For a given value $\mu \in \cP$ (the parameter set $\cP$ will be presented below), we consider the two-group neutron diffusion equations where the neutron flux $u_\mu^{\rm ex}:=\left(\phi_ {1,\mu}^{\rm ex},\phi_{2,\mu}^{\rm ex}\right)\in H^1_0(\Omega)^2$ is decomposed into the neutron flux of high energy  $\phi_ {1,\mu}^{\rm ex}\in H^1_0(\Omega)$ and thermal energy $\phi_{2, \mu}^{\rm ex}\in H^1_0(\Omega)$. This model reads as
\begin{align}
\label{eq:diffusion_equation}
&          \text{Find $\left(u^{\rm ex}_\mu:=\left(\phi^{\rm ex}_{1,\mu},\phi^{\rm ex}_{2,\mu}\right), \lambda^{\rm ex}_{\mu}\right)\in H^1_0(\Omega)^2 \times \R$ such that } \lambda^{\rm ex}_\mu \text{ is an eigenvalue with minimal modulus, and}\nonumber\\
&\left\{
\begin{aligned}
  & -\text{div}\left(D_{1,\mu}\nabla\phi_{1,\mu}^{\rm ex}\right) +   \Sigma_{11,\mu}\phi_{1,\mu}^{\rm ex} + \Sigma_{12,\mu}\phi_{2,\mu}^{\rm ex}    =\lambda_{\mu}^{\rm ex}
  F_{1,\mu}\left(\phi^{\rm ex}_{1,\mu},\phi^{\rm ex}_{2,\mu}\right)
  \quad \text{ in } \Omega ,\\
&     -\text{div}\left(D_{2,\mu}\nabla\phi_{2,\mu}^{\rm ex}\right) + \Sigma_{21,\mu}\phi^{\rm ex}_{1,\mu}  + \Sigma_{22,\mu}\phi^{\rm ex}_{2,\mu}   =
\lambda^{\rm ex}_{\mu}
F_{2,\mu}\left(\phi^{\rm ex}_{1,\mu},\phi^{\rm ex}_{2,\mu}\right)
\quad \text{ in } \Omega ,\\
&	 \phi^{\rm ex}_{i,\mu} = 0,\; \text{ on } \partial\Omega , \; i = 1,2,\\
\end{aligned}
\right.
\end{align}
supplemented with a normalization condition on $u^{\rm ex}_\mu$. Here, for all $i,j \in \{1,2\}$, all $\mu \in \cP$ and all $\phi_1, \phi_2 \in H^1_0(\Omega)$,
\begin{itemize}
\item $F_{i,\mu}\left(\phi_1,\phi_2\right) := \chi_{i,\mu}\left((\nu\Sigma_f)_{1,\mu}\phi_1 + (\nu\Sigma_f)_{2,\mu}\phi_2\right)$ ; 
\item $\chi_{i,\mu}: \Omega \to \R$ is the neutron total spectrum of group $i$;
\item $\nu_{i,\mu}: \Omega \to \R$ is  the average number of neutrons emitted per fission of group $i$;
\item $\Sigma_{fi,\mu}: \Omega \to \R$ is the fission cross section of group $i$;
\item $D_{i,\mu}: \Omega \to \R_+$ is the diffusion coefficient of group $i$;
\item 
$\Sigma_{ij,\mu}: \Omega \to \R$ with
$\Sigma_{ij,\mu} =\left\{\begin{aligned}
    \Sigma_{ti,\mu}-\Sigma_{s,0,ii,\mu} \quad \text{ if }i=j,\\ 
    -\Sigma_{s,0,ij,\mu}  \quad \text{ otherwise};
\end{aligned}\right.$
\item $\Sigma_{ti,\mu}: \Omega \to \R$ is  the total cross section of group $i$; 
\item $\Sigma_{s,0,ij,\mu}: \Omega \to \R$ is the Legendre moment of order 0 of the scattering cross section from group $i$ to group $j$.
\end{itemize}
Note that in the equations above, we used the short-hand notation $(\nu \Sigma_{f})_{i,\mu}$ to refer to the product $\nu_{i,\mu}\Sigma_{fi,\mu}$ for $i=1,2$.  
The so-called effective multiplication factor $\displaystyle k^{\rm ex}_{\mu}:=\frac{1}{\lambda^{\rm ex}_\mu}$ measures the balance between the production and {loss} of neutrons. If $k^{\rm ex}_{\mu}=1$, the nuclear chain reaction is self-sustaining; if $k^{\rm ex}_{\mu}>1$, the chain reaction is diverging; if $k^{\rm ex}_{\mu}<1$, the chain reaction vanishes.

\medskip

Let us now describe the considered parametric dependency. We introduce a partition $(\Omega_m)_{m=1}^{M}$ of the domain $\Omega$ with $M\in \N^*$ so that for all $1\leq m \leq M$, $\Omega_m$ is a domain with Lipschitz, piecewise regular boundaries.
For $i=1,2$, the coefficients $D_{i,\mu}$, $\Sigma_{ij,\mu}$, $\chi_{i,\mu}$, $(\nu\Sigma_f)_{i,\mu}$ are assumed to be piecewise regular on each domain $\Omega_m$ for $1\leq m \leq M$. 
The parameter set $\cP$ then refers to the set of values that each of these coefficients can possibly take on each subdomain $(\Omega_m)_{1\leq m \leq M}$. In other words, the choice of a parameter value $\mu \in \cP$ corresponds to a choice of the values of each of these coefficients on all the subdomains. 

\medskip

In the following, we assume that the set of admissible parameter values $\cP$ is such that all the coefficients of the model belong to $L^\infty(\Omega)$ and that 
there exists $\alpha>0$ and $0<\varepsilon<1$ such that for all $\mu \in \cP$ and all $i,j\in\{1,2\},$ $i\neq j$,  almost everywhere in $\Omega$,
\begin{align*}
 &   \alpha \leq D_{i,\mu} \quad  \text{ a.e. in }\Omega, \\
  &   \alpha \leq \Sigma_{ii,\mu} \quad \text{ a.e. in }\Omega, \\
 &   |\Sigma_{ij,\mu}|\leq \varepsilon \Sigma_{ii,\mu} \quad \text{ a.e. in }\Omega, \\
&    0\leq     (\nu\Sigma_f)_{i,\mu} \quad \text{ a.e. in }\Omega , \\
&\text{there exists } \tilde{i},\tilde{j}\in\{1,2\} \text{ such that }\chi_{\tilde{i},\mu}(\nu\Sigma_f)_{\tilde{j},\mu} \neq 0 \in L^\infty(\Omega), 
\end{align*}
so that Problem~\eqref{eq:diffusion_equation} is well-posed for all $\mu \in \cP$. The variational formulation associated to Problem~\eqref{eq:diffusion_equation} writes:
\begin{align}
\label{eq:diffusion_vf}
&          \text{Find $\left(u_\mu^{\rm ex}:=\left(\phi^{\rm ex}_{1,\mu},\phi^{\rm ex}_{2,\mu}\right), \lambda^{\rm ex}_{\mu}\right)\in (H^1_0(\Omega)^2\times \mathbb{R})$ such that } \lambda^{\rm ex}_\mu \text{ is an eigenvalue with minimal modulus and},\nonumber \\
& a_\mu(\left(\phi^{\rm ex}_{1,\mu},\phi^{\rm ex}_{2,\mu}\right),\left(\psi_1,\psi_2\right)) =\lambda^{\rm ex}_\mu b_\mu(\left(\phi^{\rm ex}_{1,\mu},\phi^{\rm ex}_{2,\mu}\right),\left(\psi_1,\psi_2\right)) \quad \text{ for all }(\psi_1,\psi_2)\in H^1_0(\Omega)^2, 
\end{align}
where for all $(\phi_1,\phi_2), (\psi_1, \psi_2)\in H^1_0(\Omega)^2$,
\begin{align*}
    a_\mu(\left(\phi_1,\phi_2\right),\left(\psi_1,\psi_2\right)) &:= \int_\Omega \left(D_{1,\mu}\nabla\phi_1\right)\cdot\nabla\psi_1 + \Sigma_{11,\mu}\phi_1\psi_1 + \Sigma_{12,\mu}\phi_2\psi_1 
    \\
    & \quad + \int_\Omega \left(D_{2,\mu}\nabla\phi_2 \right)\cdot\nabla\psi_2 + \Sigma_{21,\mu}\phi_1\psi_2  + \Sigma_{22,\mu}\phi_2\psi_2,\\
    b_\mu(\left(\phi_1,\phi_2\right),\left(\psi_1,\psi_2\right)) &:=  \int_\Omega\chi_{1,\mu}\left((\nu\Sigma_f)_{1,\mu}\phi_1 + (\nu\Sigma_f)_{2,\mu}\phi_2\right)\psi_1 \\
    \quad & + \int_\Omega\chi_{2,\mu}\left((\nu\Sigma_f)_{1,\mu}\phi_1 + (\nu\Sigma_f)_{2,\mu}\phi_2\right)\psi_2,
\end{align*}
supplemented with a normalization condition on $u_\mu^{\rm ex}$.
The associated adjoint problem then reads,
\begin{align}
\label{eq:diffusion_vf_adjoint}
&          \text{Find $\left(u_\mu^{*, {\rm ex}}:= \left(\phi_{1,\mu}^{*, {\rm ex}},\phi_{2,\mu}^{*, {\rm ex}}\right), \lambda^{\rm ex}_{\mu}\right)\in (H^1_0(\Omega)^2\times \mathbb{R})$ such that } \lambda^{\rm ex}_\mu \text{ is an eigenvalue with minimal modulus and},\nonumber \\
& a_\mu(\left(\psi_1,\psi_2\right),\left(\phi_{1,\mu}^{*, {\rm ex}},\phi_{2,\mu}^{*, {\rm ex}}\right))) =\lambda^{\rm ex}_\mu b_\mu(\left(\psi_1,\psi_2\right),\left(\phi_{1,\mu}^{*, {\rm ex}},\phi_{2,\mu}^{*, {\rm ex}}\right)) \quad \text{ for all }(\psi_1,\psi_2)\in H^1_0(\Omega)^2,
\end{align}
supplemented with a a normalization condition on $u_\mu^{*, \rm ex}$.

\subsection{The high-fidelity discretization}
\label{sec:high-fidelity_discretization}

We describe in this section the high-fidelity discretization of the continuous problem introduced in the previous section, that we consider as the reference problem in our reduced basis context. 
 Let $\mathcal{T}_{\mathcal{N}}$ be a shape-regular mesh of $\Omega$ and ${\widetilde{V}}_\mathcal{N}$ be an associated conformal finite element approximation space of dimension $\widetilde{\mathcal{N}}$. We also denote by $V_\Nh:= (\widetilde{V}_\Nh)^2$ which has dimension $\Nh = 2 \widetilde{\Nh}$. We assume that the mesh is such that the cross sections are regular on each element. 

The discrete variational formulation associated to Problem~\eqref{eq:diffusion_vf} writes
\begin{align}
\label{eq:diffusion_vf_discrete}
&          \text{Find $\left(u^\Nh_\mu:=\left(\phi_{1,\mu}^\mathcal{N},\phi_{2,\mu}^\mathcal{N}\right), \lambda_{\mu}^\Nh\right)\in (V_\mathcal{N}\times \mathbb{R})$ such that  } \lambda_{\mu}^\Nh\text{ is an eigenvalue with minimal modulus and},\nonumber \\
& a_\mu(\left(\phi_{1,\mu}^\mathcal{N},\phi_{2,\mu}^\mathcal{N}\right),\left(\psi_{1}^\mathcal{N},\psi_{2}^\mathcal{N}\right)) =\lambda_{\mu}^\Nh b_\mu(\left(\phi_{1,\mu}^\mathcal{N},\phi_{2,\mu}^\mathcal{N}\right),\left(\psi_{1}^\mathcal{N},\psi_{2}^\mathcal{N}\right)),\, \text{ for all }(\psi_{1}^\mathcal{N},\psi_{2}^\mathcal{N})\in V_\mathcal{N} = (\widetilde{V}_\Nh)^2, 
\end{align}
where $(\phi_{1,\mu}^\mathcal{N},\phi_{2,\mu}^\mathcal{N})$ satisfies a normalization condition.
We refer to~\cite{CRMATH_2021__359_5_533_0} for the {\it a priori} error analysis of Problem~\eqref{eq:diffusion_vf_discrete}.
Similarly, the discrete variational formulation associated to Problem~\eqref{eq:diffusion_vf_adjoint} reads
\begin{align}
\label{eq:diffusion_vf_discrete_adjoint}
&          \text{Find $\left(u^{*,\Nh}_\mu:=(\phi_{1,\mu}^{*,\mathcal{N}},\phi_{2,\mu}^{*,\mathcal{N}}), \lambda_{\mu}^\mathcal{N}\right)\in (V_\mathcal{N}\times \mathbb{R})$ such that  } \lambda_{\mu}^\Nh\text{ is an eigenvalue with minimal modulus and},\nonumber \\
& a_\mu(\left(\psi_{1}^\mathcal{N},\psi_{2}^\mathcal{N}\right),(\phi_{1,\mu}^{*,\mathcal{N}}(\mu),\phi_{2,\mu}^{*,\mathcal{N}})) =\lambda_{\mu}^\Nh b_\mu(\left(\psi_{1}^\mathcal{N},\psi_{2}^\mathcal{N}\right),(\phi_{1,\mu}^{*,\mathcal{N}},\phi_{2,\mu}^{*,\mathcal{N}})),\, \text{ for all }(\psi_{1}^\mathcal{N},\psi_{2}^\mathcal{N})\in V_\Nh = (\widetilde{V}_\mathcal{N})^2, 
\end{align}
where $(\phi_{1,\mu}^{*,\mathcal{N}},\phi_{2,\mu}^{*,\mathcal{N}})$ satisfies a normalization condition.

Let us denote by $(\theta^1, \cdots, \theta^\mathcal{N})$ a basis 
of $V_\mathcal{N}$. 
Problem \eqref{eq:diffusion_vf_discrete} reads as follows in matrix form. For all $\mu \in \mathcal{P}$, $i=1,2,$ let $u_\mu:=(u_{\mu,k})_{1\leq k \leq \mathcal{N}} \in \mathbb{R}^\mathcal{N}$ be the coordinates of $u_\mu^\mathcal{N}$ in the basis $(\theta^1, \cdots, \theta^\mathcal{N})$ so that 
\begin{equation}
    \label{eq:umu_decomp}
    u_\mu^\mathcal{N} := \sum_{k=1}^\mathcal{N} u_{\mu,k}\theta^k.
\end{equation}
Let us define the matrices $A_\mu:=\left( a_{\mu}(\theta^k,\theta^l)\right)_{1\leq k,l \leq \mathcal{N}}$ and $B_\mu := \left( b_\mu(\theta^k,\theta^l)\right)_{1\leq k,l \leq \mathcal{N}}$. The pair $\left( u_\mu, \lambda_\mu\right) \in \mathbb{R}^{\mathcal{N}} \times \mathbb{R}$ is then solution to 
\begin{align}
\label{eq:diffusion_vf_discrete_matricial}
A_\mu u_\mu  = \lambda_\mu  B_\mu u_\mu,
\end{align}
where $u_\mu$ satisfies a normalization condition. This is the high-fidelity eigenvalue problem of the form~\eqref{eq:hfp_lambda} we consider in the following numerical tests.

Likewise, for Problem~\eqref{eq:diffusion_vf_discrete_adjoint}, the pair $\left( u^*_\mu, \lambda_\mu\right) \in \mathbb{R}^{\mathcal{N}} \times \mathbb{R}$ is solution to 
\begin{align}
\label{eq:diffusion_vf_discrete_matricial_adjoint}
(A_\mu)^T u_\mu^* = \lambda_\mu  (B_\mu)^T u_\mu^*
\end{align}
together with a normalization condition on $u_\mu^*$. Here,  $u_\mu^* = (u_{\mu,k}^*)_{1\leq k \leq \Nh} \in \R^\Nh$ is the vector of coordinates of the function $u_\mu^{*,\Nh}$ in the basis $(\theta^1, \ldots, \theta^\Nh)$, i.e. 
$$
u_\mu^{*,\Nh}= \sum_{k=1}^{\Nh}u_{\mu,k}^* \theta^k.
$$
Problem~(\ref{eq:diffusion_vf_discrete_matricial_adjoint}) is the adjoint high-fidelity eigenvalue problem of the form~\eqref{eq:hfp_lambda_adjoint} that we consider in the following numerical tests.

\subsection{Affine decomposition of the coefficients}
\label{sec:affine_decomposition}

In the following numerical tests, the domain $\Omega$ is chosen as $[0,L]^2$ for some $L>0$. We introduce a partition $(\Omega_k)_{k=1}^K$ of the domain $\Omega$ and the {parameter functions}
entering in the definition of Problem~\ref{eq:diffusion_equation} are assumed to be piecewise constant on each $\Omega_k$ for $1\leq k \leq K$.
{The parameter $\mu$ is thus  a $K$-dimensional vector of either scalars or vectors (containing macro-parameters such as the material, the burn up, the fuel temperature, or the boron concentration for example), which allows to set the values  of the coefficients $D_1, \Sigma_{11}, \Sigma_{12}, D_2, \Sigma_{21}, \Sigma_{22}, \chi_1, \chi_2, \Sigma_{f1}, \Sigma_{f2}$ on the domain $\Omega_k$, for each $1\leq k \leq K$.
}
We remark that for all $\mu\in \mathcal{P}$, the matrices $A_{\mu}$ and $B_{\mu}$ write
$$
       A_{\mu} = \displaystyle \sum_{k=1}^{K}  \sum_{p=1}^{6} f_p(\mu_k)A_{k,p} + M_{bc} \text{  and  } 
       B_{\mu} = \displaystyle\sum_{k=1}^{K} \sum_{q=1}^{4} g_q(\mu_k)B_{k,q},
$$
where $f(\mu_k)$ and $g(\mu_k)$ are the 
vectors 
defined by
    \begin{align*}
        &   f(\mu_k) = \left(D_1(\mu_k), \Sigma_{11}(\mu_k), \Sigma_{12}(\mu_k), D_2(\mu_k), \Sigma_{21}(\mu_k), \Sigma_{22}(\mu_k) \right) \\
        &   g(\mu_k) = \left(\left(\chi_1\nu\Sigma_{f1}\right)(\mu_k),\left(\chi_1\nu\Sigma_{f2}\right)(\mu_k),\left(\chi_2\nu\Sigma_{f1}\right)(\mu_k),\left(\chi_2\nu\Sigma_{f2}\right)(\mu_k)\right),
    \end{align*}
$A_{k,p}$ and $B_{k,q}$ ($1\leq k\leq K$, $1\leq p\leq 6$, $1\leq q\leq 4$) are parameter-independent $\Nh\times \Nh$ matrices, and $M_{bc} \in \mathbb{R}^{\Nh \times \Nh}$ is a parameter-independent matrix which stems from the discretization of the boundary condition. As a consequence, all these matrices can be pre-computed in order to efficiently assemble the matrices $A_\mu$ and $B_\mu$ online, and estimate the residuals $R_N(\mu)$ and $R^*_N(\mu)$ as we now explain.

Thanks to the affine decomposition of the matrices $A_{\mu}$ and $B_{\mu}$ above, the residual norm is easily computable online, as it only requires algebraic operations over {vectors of the size of the (small) reduced basis, which is $N$.} Indeed, let $(\xi_1,\ldots,\xi_N)$ be an orthonormal basis of the chosen reduced space {for the scalar product $\langle\cdot,\cdot\rangle$}, and let $V_N \in \mathbb{R}^{\Nh \times N}$ be the matrix containing the coordinates of the basis $(\xi_1, \ldots, \xi_N)$ {in the canonical basis of $\R^\Nh$}. For $1\leq k,l \leq K, 1\leq p,p'\leq 6$, $1\leq q,q'\leq 4$, we define, in the offline stage, the reduced matrices of dimension $N\times N$ as follows:

\begin{align*}
    &   D^N_{k,l,p,p'} = V_N^tA_{k,p}^t \mathbb{X}^{-1}A_{l,p'}V_N \\
    &   E^N_{k,l,p,q} = V_N^tA_{k,p}^t \mathbb{X}^{-1}B_{l,q}V_N \\
    &   F^N_{k,l,q,q'} = V_N^tB_{k,q}^t \mathbb{X}^{-1}B_{l,q'}V_N \\
    &   D^N_{bc,k,p} = V_N^tM_{bc}^t \mathbb{X}^{-1}A_{k,p}V_N \\
    &   E^N_{bc,k,q} = V_N^tM_{bc}^t \mathbb{X}^{-1}B_{k,q}V_N \\
    &   F^N_{bc} = V_N^tM_{bc}^t \mathbb{X}^{-1}M_{bc}V_N,
\end{align*}
where $\mathbb{X}$ stands for the Gram matrix {of size $\Nh\times \Nh$} for the considered scalar product $\langle\cdot,\cdot\rangle$, commonly called the mass matrix, {and $A^t$ denotes the transpose of the matrix $A$}.
Then, in the online stage, for a given parameter $\mu$, we can assemble the residual norm as
\begin{align*}
    \|R_N(\mu)\| := \|\left(B_{\mu}-k_{\mu,N} A_{\mu}\right)u_{\mu,N}\| = \sqrt{c_{\mu,N}^tG_{\mu,N}c_{\mu,N}},
\end{align*}
with
\begin{align*}
    G_{\mu,N} =& |k_{\mu,N}|^2 \left(\displaystyle \sum_{k,l=1}^{K}  
    \sum_{p,p'=1}^{6}  f_p(\mu_k)f_{p'}(\mu_l)D_{k,l,p,p'}^N  + 
    \sum_{k=1}^{K} \sum_{p=1}^{6}
    f_p(\mu_k) \left(D^N_{bc,k,p} + (D^N_{bc,k,p})^t\right) + F^N_{bc}\right) \\
    &- k_{\mu,N} \displaystyle \left( \sum_{k,l=1}^{K}\sum_{p=1}^{6} \sum_{q=1}^{4}
    f_p(\mu_k)g_q(\mu_l)\left(E_{k,l,p,q}^N+(E_{k,l,p,q}^N)^t\right)\right) + \sum_{k=1}^{K} \sum_{q=1}^{4}  g_q(\mu_k)\left(E^N_{bc,k,q}+(E^N_{bc,k,q})^t\right) \\
    &+ \displaystyle \sum_{k,l=1}^{K}\sum_{q,q'=1}^{4}  g_q(\mu_k)g_{q'}(\mu_l)F_{k,l,q,q'}^N .
\end{align*}
A similar construction is readily possible for $\|R^*_N(\mu)\|$.

\subsection{Eigenvalue solver}
\label{sec:eigenvalue_solver}

The eigenvalue solver, for both high-fidelity and reduced-order models, relies on the inverse power method given in Algorithm \ref{algorithm:power_it}. In practice, this algorithm is run with relative error tolerances set to $\tau_u=10^{-6}$ and $\tau_{\lambda}=10^{-7}$.

\begin{algorithm}
\caption{\textsc{Inverse power method - solve $\mathtt{A} u = \lambda \mathtt{B}u$}}
\label{algorithm:power_it}
\begin{algorithmic}
    \STATE{\bf Input:} $\mathtt{A} \in \mathbb{R}^{M\times M}$, $\mathtt{B}\in\mathbb{R}^{M\times M}$, $\tau_u$: acceptance criterion for the eigenvector, $\tau_{\lambda}$: acceptance criterion for the eigenvalue
    \STATE{Choose a random positive unit vector $u_0$ and $k_0\neq 0$}
    \STATE{Set $i=0$ and STOP=\FALSE}
    \WHILE{(STOP==\FALSE)}
    \STATE{Solve $\mathtt{A}v_{i+1} = \mathtt{B}u_i$}
    \STATE{$u_{i+1} = \dfrac{v_{i+1}}{\left\|v_{i+1}\right\|}$}
    \STATE{$k_{i+1} = \langle v_{i+1}, u_i \rangle$}
    \STATE{STOP=  $\Bigg[ \dfrac{\left\|u_{i+1}-u_i\right\|}{\left\|u_i\right\|} \le \tau_u$ \AND $\dfrac{\left|k_{i+1}-k_i\right|}{\left|k_i\right|} \le \tau_{\lambda} \Bigg]$ }
    \STATE{$i=i+1$}
    \ENDWHILE
    \STATE{\bf Output:} $(u,\lambda)=\left(u_{i}, \frac{1}{k_{i}}\right)$
\end{algorithmic}
\end{algorithm}

The direct high-fidelity eigenvalue problem~\eqref{eq:hfp_lambda} is solved by applying Algorithm \ref{algorithm:power_it} with $\mathtt{A} = A_{\mu}$ and $\mathtt{B} = B_{\mu}$. Likewise, the  adjoint eigenvalue problem~\eqref{eq:hfp_lambda_adjoint} is solved by applying Algorithm \ref{algorithm:power_it} with $\mathtt{A} = A^T_{\mu}$ and $\mathtt{B} = B^T_{\mu}$. The resolutions of the reduced eigenvalue problems are performed similarly.

\subsection{Numerical tests}
\label{sec:numerical_tests}

The aim of this section is to illustrate the numerical behaviour of the reduced basis method and the proposed {\it a posteriori} error estimators on two different numerical test cases. Let us introduce the notation 
\begin{align*}
    &   e_N^k(\mu) = \left|k_{\mu}-k_{\mu,N}\right|,\quad e_N^{k,rel}(\mu) = \dfrac{\left|k_{\mu}-k_{\mu,N}\right|}{|k_{\mu}|},\\
    &   e_N^u(\mu) = \|u_{\mu}-u_{\mu,N}\|_{\ell^2},\quad 
    e_N^{u,rel}(\mu) = \dfrac{\|u_{\mu}-u_{\mu,N}\|_{\ell^2}}{\|u_{\mu}\|_{\ell^2}},\quad e_{N,L^2}^{u,rel}(\mu) = \dfrac{\|u_{\mu}-u_{\mu,N}\|_{L^2}}{\|u_{\mu}\|_{L^2}},\\
    &   e_N^{u^*}(\mu) = \|u^*_{\mu}-u^*_{\mu,N}\|_{\ell^2},
\end{align*}
where the $\ell^2$ norm is the Euclidean norm, and $L^2$ refers to the $L^2$ functional norm applied to the functions in the space $V_\Nh$ built from the vectors in $\R^{\mathcal N}$ through~\eqref{eq:umu_decomp}.
Moreover, we denote by $t_{\rm HF}$ and $t_{\rm RB}$ the  mean computational times for one run (for a given parameter) of the high-fidelity and reduced solvers respectively.

\subsubsection{Test case 1: 2D two-group toy example}\label{sec:test_case_1}

The reduced basis method is first run on a simple test case where $L=60$ (we use here reduced units)  
modeled with $\mathcal{N}= 2 \times 841$ degrees of freedom along $K=4$ subdomains. Figure~\ref{fig:toycore} shows the mesh used for the test case as well as the decomposition of $\Omega$ into four subdomains. Here, we set $B_{\mu}=I$ for all $\mu \in \mathcal P$.
\begin{figure}[h!]
    \centering
    \includegraphics[scale=0.5]{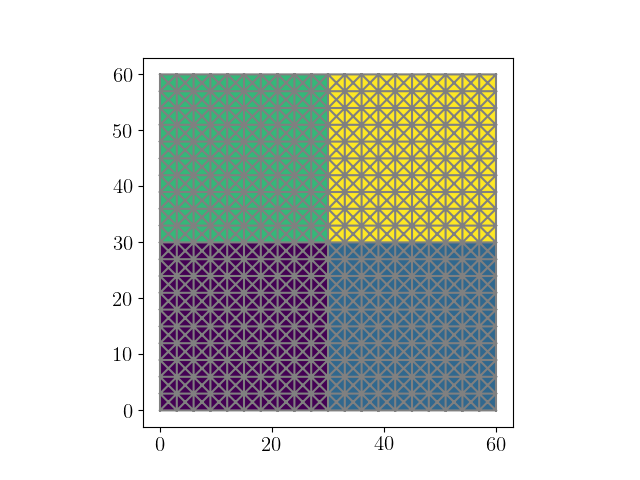}
    \caption{Domain of calculation for the two-group toy example with its associated mesh}
    \label{fig:toycore}
\end{figure}

The training and test sets $\mathcal P_{\rm train}$ and $\mathcal P_{\rm test}$ are constructed using the following random sampling scheme: in each subdomain $\Omega_k$, for $1\leq k \leq K$, the values of the coefficients are independently distributed according to the following laws:
\begin{itemize}
    \item $\Sigma_{s,0,ij}: $ uniform law on $[0,0.15]$ , $1\leq i,j\leq 2$;
    \item $\Sigma_{t1}$ and $\Sigma_{t2}$: uniform law on $[2(\Sigma_{s,0,12}+\Sigma_{s,0,21}),0.7]$;
    \item $\displaystyle D_i=\frac{1}{3\Sigma_{ti}}$, $i=1,2$;
    \item $\chi_i\nu\Sigma_{fj}=\delta_{ij}$, $1\leq i,j\leq 2$.
\end{itemize}
The coefficients are chosen so that the coercivity of Problems~\eqref{eq:diffusion_vf} and~\eqref{eq:diffusion_vf_adjoint} are ensured.
The parametric spaces $\mathcal{P}_{\text{train}}$ and $\mathcal{P}_{\text{test}}$ are selected following the random sampling procedure described above so that $\#\mathcal{P}_{\text{train}} = 300$, $\#\mathcal{P}_{\text{test}} = 50$ and $\mathcal{P}_{\text{train}} \cap \mathcal{P}_{\text{test}} = \emptyset$.
In the offline stage, the greedy algorithm is performed using the \textit{a posteriori} estimator $$\Delta_N(\mu) =\eta_N^k(\mu)$$ defined in~\eqref{eq:etaNK} for all $\mu \in \mathcal P$ (in other words, we choose here $\bar{C}_N^k(\mu)=1$ for all $\mu$, following the notation~\eqref{eq:Cbar}).

The left part of Figure \ref{fig:toycore_mean_errors} depicts the fast convergence of the reduced basis method with respect to the size of the reduced space. The relative errors on the eigenfunctions $e_N^{u,rel}(\mu)$ and $e_{N,L^2}^{u,rel}(\mu)$ follow the same trend. The relative error $e_N^{k,rel}(\mu)$ between the high-fidelity solution and the reduced basis solution on the multiplication factor $k_\mu$ reaches the order of $10^{-5}$
for $N=100$. Moreover, this error decreases by 4 orders of magnitude from $N=10$ to $N=100$. As expected, the error on the eigenvalue decreases twice faster than the error on the eigenvector.
Moreover, we checked that the value of the a posterior error estimator $\eta_N^k(\mu)$ stays below $10^{-12}$ for the selected parameters, as expected.

\begin{figure}[h!]
    \centering
    \includegraphics[scale=0.5]{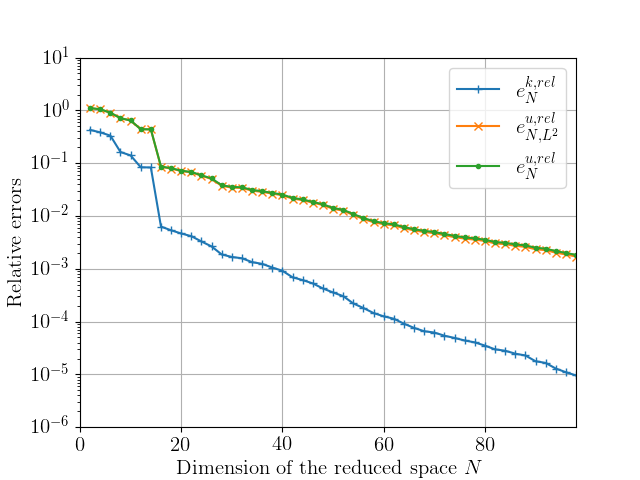}
    \includegraphics[scale=0.5]{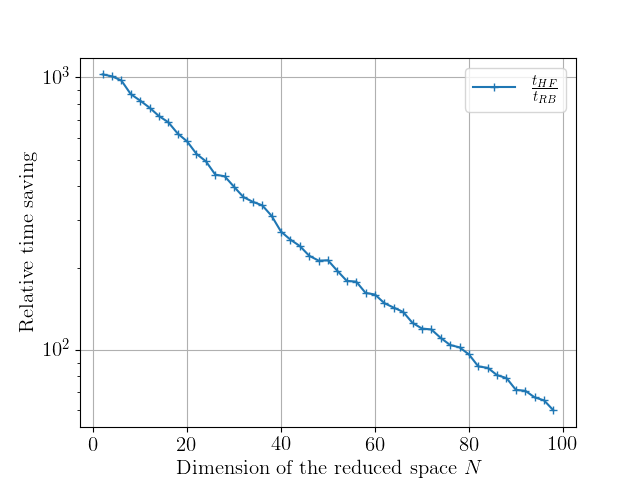}
    \caption{(Left) Mean relative errors over $\mathcal{P}_{\text{test}}$; (Right) Relative time saving factor $\dfrac{t_{\rm HF}}{t_{\rm RB}}$ as a function of the dimension of the reduced space $N$.
    }
    \label{fig:toycore_mean_errors}
\end{figure}

In terms of computational time, the right part of Figure \ref{fig:toycore_mean_errors} shows that, in the chosen setting, while the high-fidelity solution is computed in about 5.8s, the reduced solution is computed within up to 0.09s, which is overall 60 up to 115 times faster than the high-fidelity solver to obtain a relative error of order $10^{-4}$ to $10^{-5}$ on the eigenvalue.

It is also interesting to look at the behavior of the implemented \textit{a posteriori} error estimators. The relation between the error $|k_N(\mu)-k(\mu)|$ and the  estimator $\Delta_N(\mu) =\eta_N^k(\mu)$ we used here can be first analyzed by looking at the prefactor $C_N^k(\mu)$, defined in~\eqref{eq:C_N}. The  value of $C_N^k(\mu)$ on the test set $\mathcal P_{\rm test}$ is presented in Figure~\ref{fig:toycore_prefactor}. In that particular case, we 
fall into the framework developed in
Section \ref{sec:perturbation}. Indeed, when we compute the perturbation magnitude $\varepsilon_{\mu}$ as
\begin{equation}
    \varepsilon_{\mu} = \dfrac{\left\|\frac{A_{\mu}-A_{\mu}^T}{2}\right\|}{\left\|\frac{A_{\mu}+A_{\mu}^T}{2}\right\|},
\end{equation}
we observe that $\varepsilon_\mu$ varies between $3\times 10^{-7}$ and $3\times 10^{-6}$ for $\mu \in \mathcal{P}_{\text{test}}$.
 Therefore, we expect $C_N^{k,{\rm sym}}(\mu)$ defined in~\eqref{eq:CNksym} to be a a good approximation of $C_N^k(\mu)$. Unfortunately, this is not always the case as we observe on the left plot of Figure~\ref{fig:toycore_prefactor}. 
Actually, in the cases where the prefactors differ a lot, we observe that condition~\eqref{gap_assumption2} in Proposition~\eqref{prop:prefactor_upper_bound} is not satisfied, which explains why the perturbative expansions may not be sharp.

\begin{figure}[h!]
    \centering
    \includegraphics[scale=0.5]{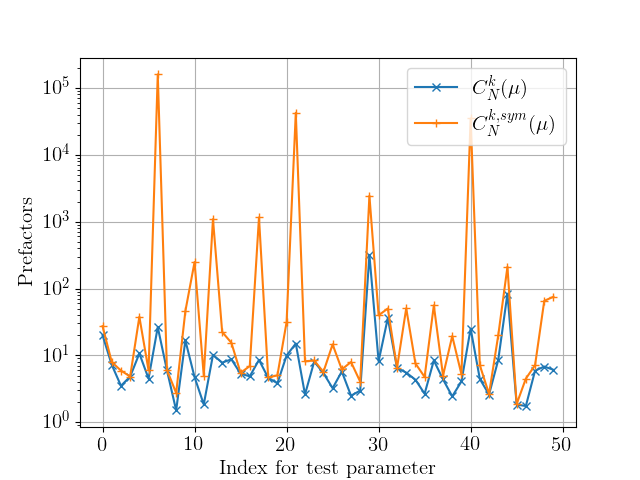}
     \includegraphics[scale=0.5]{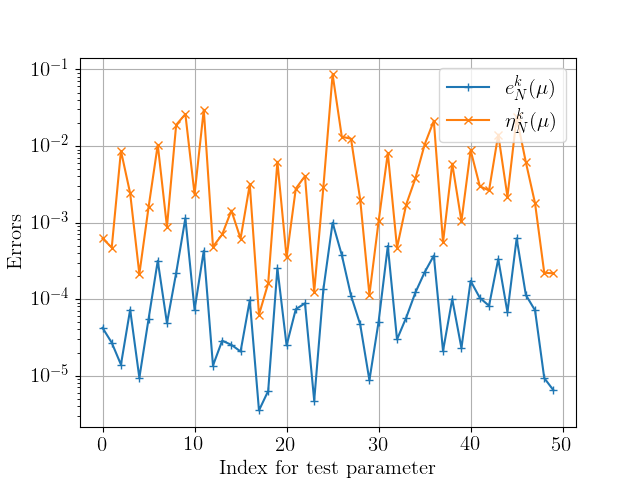}
    \caption{(Left) Variations of the prefactor $C_N^k(\mu)$ and $C_N^{k, {\rm sym}}(\mu)$ over $\mathcal{P}_{\text{test}}$ for $N=100$. \\
    (Right) Parametric variations of the real eigenvalue error $e_N^k(\mu)$ (in blue) and the associated \textit{a posteriori} error estimator $\eta_N^k(\mu)$ (in orange) over $\mathcal{P}_{\text{test}}$, for $N=100$. }
    \label{fig:toycore_prefactor}
\end{figure}

Figure \ref{fig:toycore_mean_errors_vs_APE} compares the behavior of the simple \textit{a posteriori} error estimators $\|R_N(\mu)\|$, $\|R_N^*(\mu)\|$ and $\eta_N^k(\mu)$ defined in~\eqref{eq:etaNK}, with the corresponding errors $e_N^u(\mu)$, $e_N^{u^*}(\mu)$, and $e_N^k(\mu)$ over the dimension of the reduced space.
The plots of the true errors and the corresponding estimators are parallel from $N \geq 20$ which confirms the similar convergence rate for the computed \textit{a posteriori} estimators as the associated real errors. 
Actually, the quantity $\eta_N^k(\mu)$ seems to be a reliable and efficient \textit{a posteriori} estimate of the true error up to roughly a constant multiplicative factor over a large range of parameter values, as Figure \ref{fig:toycore_prefactor} illustrates. 

In terms of absolute value, for $N=100$, the estimator $\eta_N^k(\mu)$ for the multiplication factor is about $10^{-2}$ while the true error is approximately $10^{-4}$: this illustrate the importance of introducing prefactors $\bar{C}_N^k(\mu)$, $\bar{C}_N^u(\mu)$ and $\bar{C}_N^{u^*}(\mu)$ to estimate the true errors, see~\eqref{eq:Cbar}. This is important in particular in order to stop the greedy procedure when the real error is below a given threshold. This will be discussed below.

\begin{figure}[h!]
    \centering
    \includegraphics[scale=0.35]{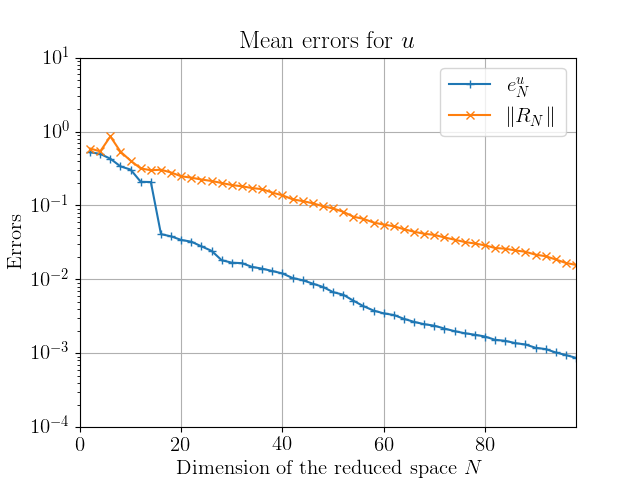}
    \hspace*{-.7cm}
    \includegraphics[scale=0.35]{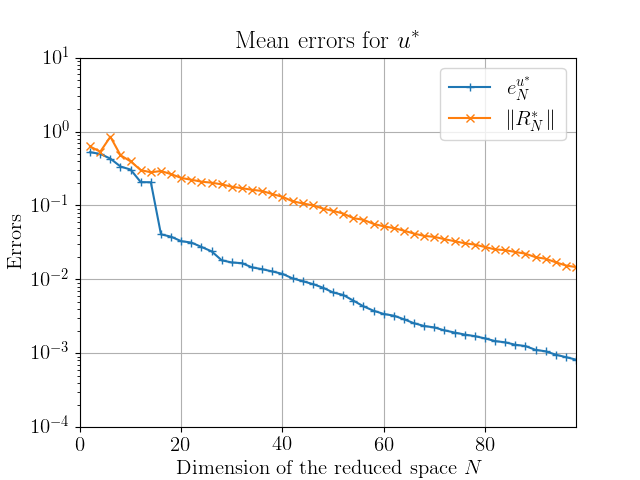}
    \hspace*{-.7cm}
    \includegraphics[scale=0.35]{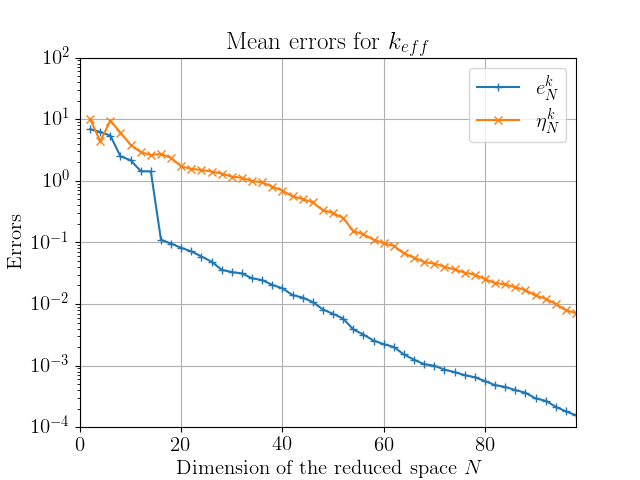}
    \caption{Mean values for errors and associated \textit{a posteriori} error estimators over $\mathcal{P}_{\text{test}}$. (Left) $e_N^u$ and $\|R_N\|$; (Middle) $e_N^{u^*}$ and $\|R_N^*\|$; (Right) $e_N^k$ and $\eta_N^k$.}
    \label{fig:toycore_mean_errors_vs_APE}
\end{figure}

\subsubsection{Test case 2: 2D two-group "minicore" problem}

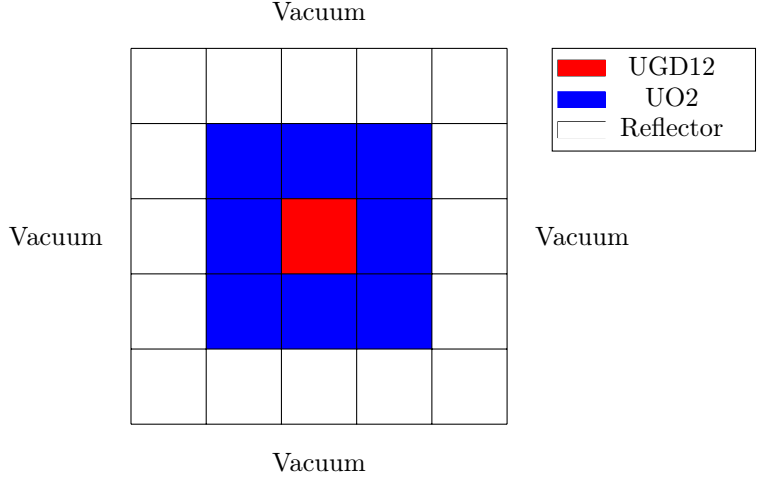
\begin{figure}[htbp]
    \begin{center}
\begin{tikzpicture}[scale=5./25.]
\fill[white] (0,0) rectangle (25.,25.);
\fill[blue] (5.,5.) rectangle (20.,20.);
\fill[red] (10.,10.) rectangle (15.,15.);
\draw (0,0) grid[xstep=5.,ystep=5.](25.,25.);
\draw[black](28.38,23.17) rectangle (31.512,24.2159);
\fill[red] (28.38,23.17) rectangle (31.512,24.2159);
\node at (35.953,23.851925) {UGD12 };
\draw[blue](28.38,21.17-0.1) rectangle (31.512,22.2159-0.1);
\fill[blue] (28.38,21.17-0.1) rectangle (31.512,22.2159-0.1);
\node at (35.953,21.851925-0.1) {UO2 };
\draw[black](28.38,19.17-0.1) rectangle (31.512,20.2159-0.1);
\fill[white] (28.38,19.17-0.1) rectangle (31.512,20.2159-0.1);
\node at (35.953,19.851925-0.1) {Reflector };
\draw[black](28.,18.17) rectangle (41.512,25.);
\node at(0.5*25.,-0.1*25.) { Vacuum};
\node at(-0.2*25.,0.5*25.) { Vacuum};
\node at(0.5*25.,+1.1*25.) { Vacuum};
\node at(+1.2*25.,0.5*25.) { Vacuum};
\end{tikzpicture}
\end{center}
\caption{Median cross-sectional view of the \textit{MiniCore} ($z = 234.36$ cm)}
    \label{fig:minicore}
\end{figure}

We now provide a second, more challenging, test case called \textit{minicore}.
The core is modeled as a square of side length $L=107.52$ cm. As Figure \ref{fig:minicore} shows, it is made up of $K=25$ assemblies (1 fuel assembly composed of a mix of uranium dioxyde and Gadolinium oxyde denoted UGD12 + 8 fuel assemblies composed of uranium dioxyde labeled UO2 + 16 radial reflector assemblies named REFR), 
 each one being $21.504$ cm long. 
 It is discretized into $\mathcal{N}=2602$ degrees of freedom. Here, there holds $B_{\mu} \neq I$, and the Dirichlet boundary condition in Problem~\eqref{eq:diffusion_equation} is replaced by a Robin condition called void boundary condition which writes
\begin{align*}
     D_i(r,\mu) \nabla\phi_i(r,\mu).\Vec{n} + \dfrac{1}{2}\phi_i(r,\mu)=0 \quad \text{ on }
     \partial \Omega,
     \quad 1\leq i \leq 2,
\end{align*}
where $\Vec{n}$ is the outward unit normal vector to 
$\partial \Omega$.

In this test case, the parameter $\mu$
stands for five parameters which determine all the physical parameters entering~\eqref{eq:diffusion_equation}.
More precisely, by recalling the partition $(\Omega_k)_{k=1}^K$ of the domain $\Omega$, the parameter set $\mathcal{P}$ is the $5K$ dimensional vector space 
\begin{equation*}
    \mathcal{P} = \left\{ \mu = \left(\mu_1,\ldots,\mu_K\right),\, \forall 1\le k\le K,\, \mu_k \in \mathbb{R}^5 \right\}, 
\end{equation*}
such that $\mu_k$ contains the following information attached to the subdomain $\Omega_k$: 
\begin{itemize}
    \item the nature of the material in $\Omega_k$;
    \item the burnup value, in MWd/ton;
    \item the fuel temperature, in K;
    \item the boron concentration, in particle per million (ppm);
    \item the moderator density. 
\end{itemize}
The parametric sets $\mathcal{P}_{\text{train}}$ and $\mathcal{P}_{\text{test}}$ are randomly generated in $\mathcal{P}$ such that
$$
       \#\mathcal{P}_{\text{train}} = 1000, \qquad
       \#\mathcal{P}_{\text{test}} = 50, \text{ and }\qquad
       \mathcal{P}_{\text{train}} \cap \mathcal{P}_{\text{test}} = \emptyset .
$$

Regarding the offline stage, in order to avoid any stability issue, a POD procedure over a reduced space of dimension 10 (generated from 5 direct plus 5 adjoint eigenvectors snapshots) is used to initialize the greedy  procedure (see Algorithm~\ref{algorithm:alg2}). Then, the greedy procedure is continued using the \textit{a posteriori} estimator $\|R_N\|+\|R^*_N\|$, as the quantity of interest here is the two-group flux $(\phi_1^\mathcal{N},\phi_2^\mathcal{N})$ as well as its adjoint $(\phi_1^{*,\mathcal{N}},\phi_2^{*,\mathcal{N}})$.

The left part of Figure~\ref{fig:minicore_mean_errors} depicts mean relative errors $e_N^{k,{\rm rel}}, e_{N,L^2}^{u,{\rm rel}}$, and $e_N^{u,{\rm rel}}$ as a function of the dimension of the reduced basis. 
The relative error on the multiplication factor is of the order of $10^{-5}$ for $N = 80$. 
Typically, as the left part of Figure~\ref{fig:minicore_flux1} shows, for a certain $\mu_0 \in \mathcal{P}$ and for $N=100$, the maximum error on the associated first-group flux does not exceed $3.2 \times 10^{-4}$; as for the second group, the right part of Figure~\ref{fig:minicore_flux1} shows that the flux error is locally gathered in an area of low flux, quite far from the hot spot.

\begin{figure}[h!]
    \centering
    \includegraphics[scale=0.5]{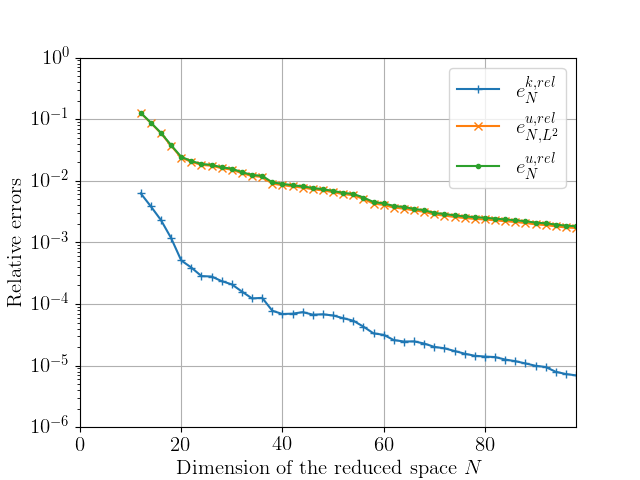}
    \includegraphics[scale=0.5]{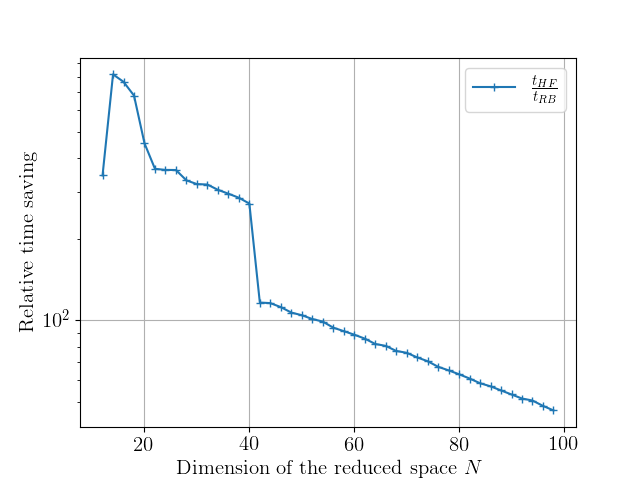}
    \caption{(Left) Mean relative errors over $\mathcal{P}_{\text{test}}$; (Right) Relative time saving factor $\dfrac{t_{\rm HF}}{t_{\rm RB}}$ as a function of the dimension of the reduced space $N$.
    }
    \label{fig:minicore_mean_errors}
\end{figure}

\begin{figure}[h!]
    \centering
    \includegraphics[scale=0.5]{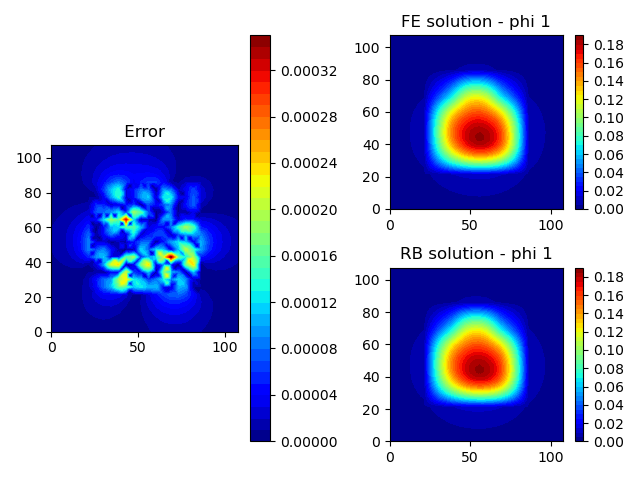}
    \includegraphics[scale=0.5]{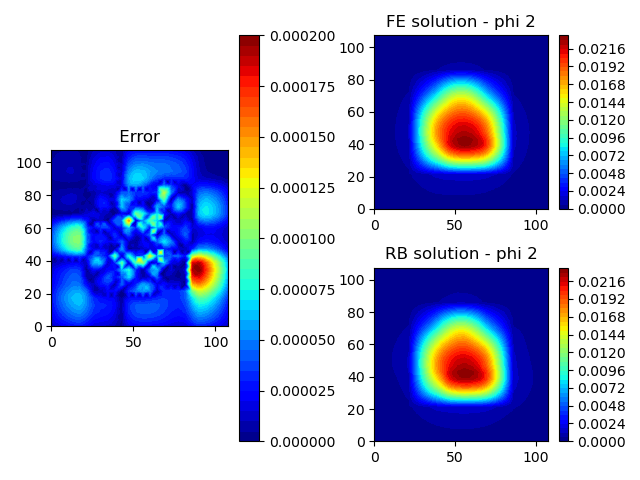}
    \caption{(Left) Plots of the first energy group of high-fidelity (upper right) and reduced (lower right) solutions $u_{\mu_0}$ and $u_{\mu_0,N}$, and their error (left) $\left|u_{\mu_0}-u_{\mu_0,N}\right|$, for $N=100$ and for $\mu_0 \in \mathcal{P}_{\text{test}}$; 
    (Right) Plots of the second energy group of high-fidelity (upper right) and reduced (lower right) solutions $u_{\mu_0}$ and $u_{\mu_0,N}$, and their error (left) $\left\|u_{\mu_0}-u_{\mu_0,N}\right\|$, for $N=100$ and for $\mu_0 \in \mathcal{P}_{\text{test}}$} \label{fig:minicore_flux1}
\end{figure}

Importantly, the reduced method enables the solution to be computed faster than the high-fidelity approach, which typically takes about 4.56 s to be computed for the present test case.
The right part of Figure~\ref{fig:minicore_mean_errors} illustrates that the relative saving time factor is a decreasing function of the dimension of the reduced space~$N$, and exhibits a large computational gain compared to the high-fidelity solver.
It is observed that for a relative error on $k_{\rm eff}$ ranging from $10^{-4}$ to $10^{-6}$, the reduced solution can be obtained with a computational time from 50 up to 300 times smaller than the high-fidelity solution. 

\begin{figure}[h!]
    \centering
    \includegraphics[scale=0.35]{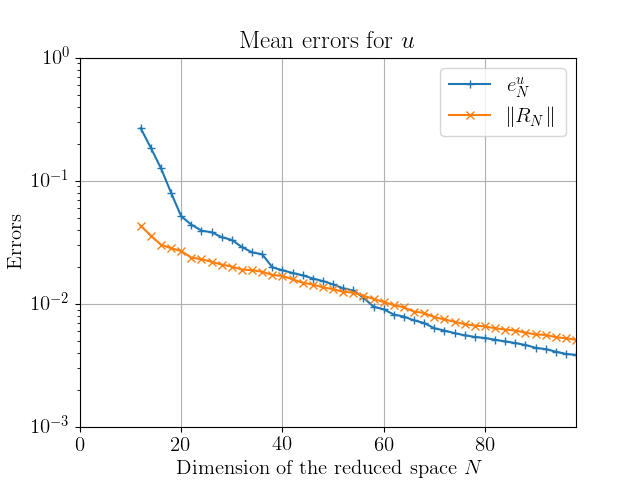}
     \hspace*{-.7cm}
    \includegraphics[scale=0.35]{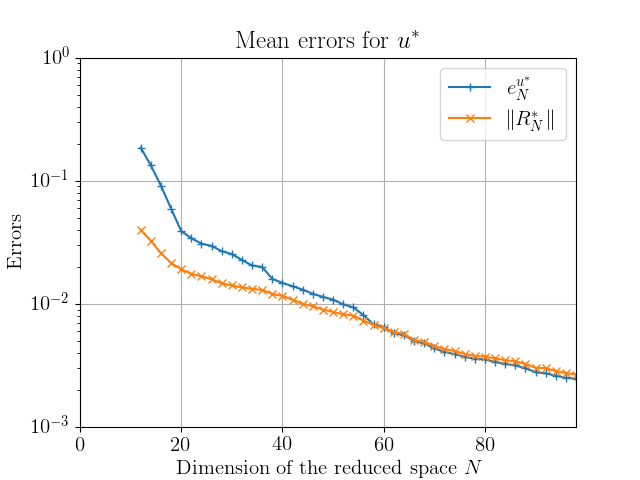}
     \hspace*{-.7cm}
    \includegraphics[scale=0.35]{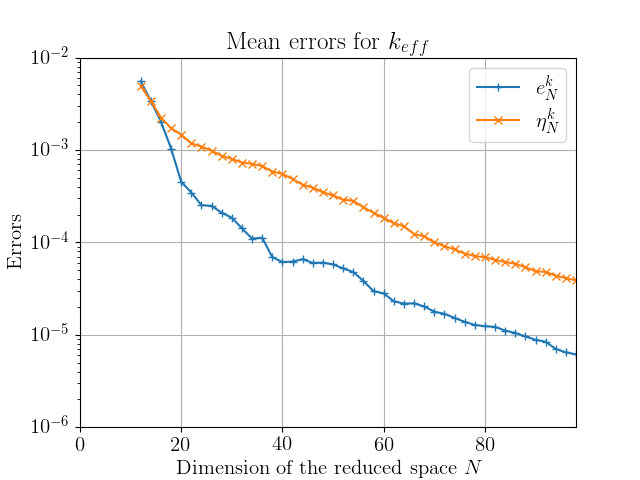}
    \caption{Mean values for errors and associated \textit{a posteriori} error estimators over $\mathcal{P}_{\text{test}}$. \\ (Left) $e_N^u$ and $\|R_N\|$; (Middle) $e_N^{u^*}$ and $\|R_N^*\|$; (Right) $e_N^k$ and $\eta_N^k$.}
    \label{fig:minicore_mean_errors_vs_APE}
\end{figure}

We now study the certification of the method performed by the \textit{a posteriori} estimator. Figure~\ref{fig:minicore_mean_errors_vs_APE} shows that, although the residuals display similar values as those for the real eigenvector errors, for the eigenvalue, the order of magnitude of the \textit{a posteriori} estimator is roughly 10 times larger than the real error, for $N\ge 30$.  Despite the fairly good parametric variations of the estimate, illustrated by Figure~\ref{fig:minicore_mean_errors_vs_APE_param}, the gap between real error and estimator must be corrected in order to implement a relevant stopping criterion in the greedy algorithm. This points out a certain variation of the prefactor $C_N^k(\mu)$ over the dimension of the reduced space $N$.
In order to bring a correction to the model, the practical efficiency of the estimator proposed in Section~\ref{sec:practical} is computed. 
The right plot of Figure~\ref{fig:minicore_mean_errors_vs_APE_param} shows that the efficiency $\mathcal{E}_N^k$ defined in~\eqref{eq:ENk} levels off for $N=100$ at the order of magnitude of $10^{-1}$, and does not depend too much on the parameter $\mu$. 
Therefore, we propose to apply the procedure outlined in 
Section~\ref{sec:practical} to build a posteriori error estimators of the form~\eqref{eq:Cbarmax_delta}, with constants $\overline{C}_N^k$, $\overline{C}_N^u$ and $\overline{C}_N^{u^*}$ approximated by~\eqref{eq:Cbarmax}. This requires to
choose a set $\mathcal{P}_{\text{pref}}$, that we  randomly chose in $\mathcal{P}$ such that
$$
     \mathcal{P}_{\text{pref}} \subset \mathcal{P}, \qquad
       \#\mathcal{P}_{\text{pref}} = 10, \text{ and } \qquad
       \mathcal{P}_{\text{pref}} \cap \mathcal{P}_{\text{train}} \cap \mathcal{P}_{\text{test}} = \emptyset .
$$
As a result of this procedure, Figure~\ref{fig:minicore_mean_errors_vs_APE_updated2} shows that the order of magnitude of the modified estimator corresponds to the one of the real error, showing that the new \textit{a posteriori} estimator tends to be an optimal stopping indicator.

\begin{figure}[h!]
    \centering
    \includegraphics[scale=0.5]{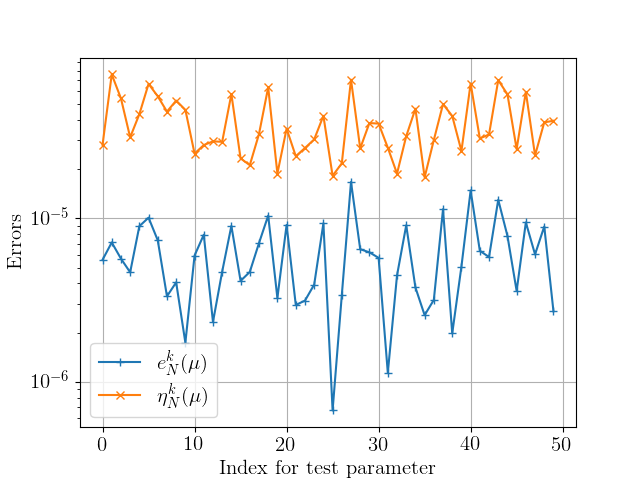}
    \includegraphics[scale=0.5]{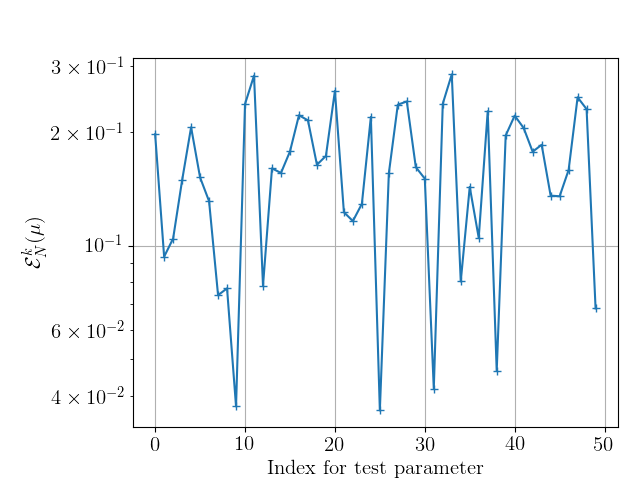}
    \caption{(Left) Parametric variations of the real eigenvalue error $e_N^k(\mu)$ (in blue) and its associated \textit{a posteriori} error estimator $\eta_N^k(\mu)$ (in orange) over $\mathcal{P}_{\text{test}}$, for $N=100$; (Right) Parametric variations of the practical efficiency $\mathcal{E}_{N}^k(\mu)$ over $\mathcal{P}_{\text{test}}$, for $N=100$.}
    \label{fig:minicore_mean_errors_vs_APE_param}
\end{figure}
\begin{figure}[h!]
    \centering
    \includegraphics[scale=0.35]{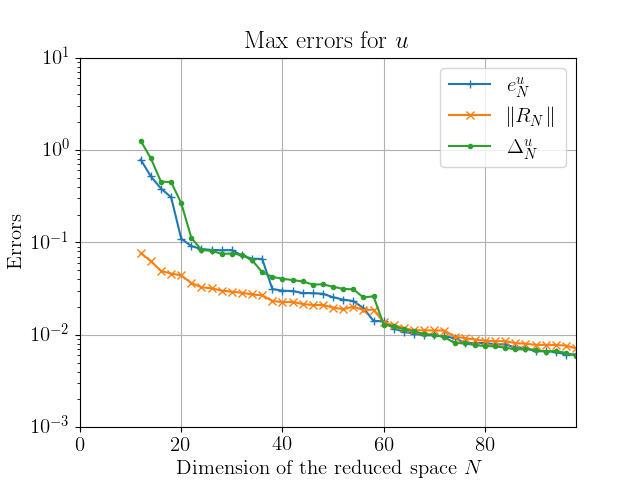}
     \hspace*{-.7cm}
    \includegraphics[scale=0.35]{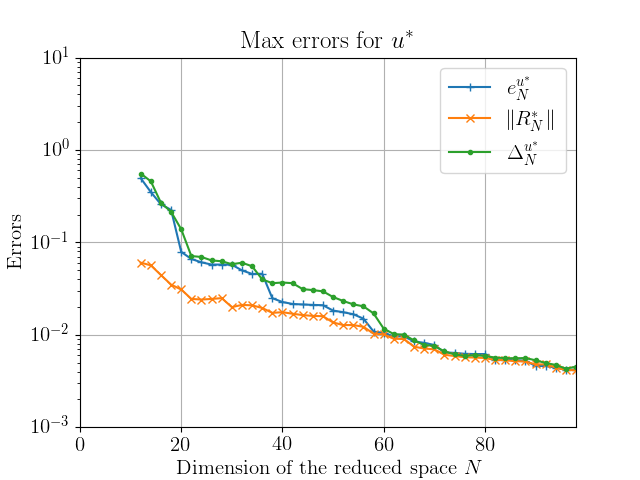}
     \hspace*{-.7cm}
    \includegraphics[scale=0.35]{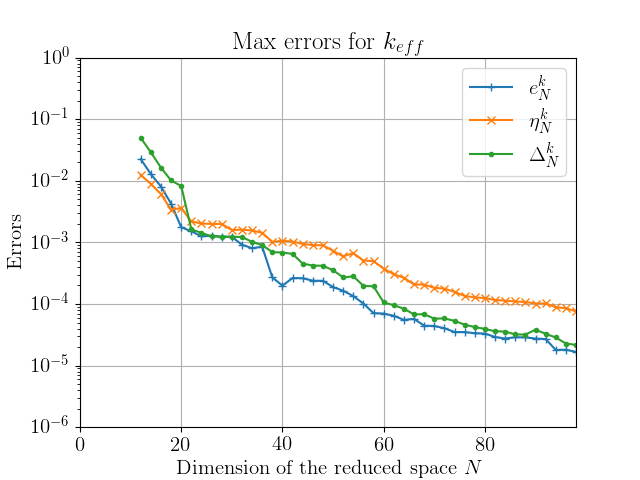}
    \caption{Maximum values for errors and associated \textit{a posteriori} error estimators over $\mathcal{P}_{\text{test}}$. Upper left: $e_N^u$, $\|R_N\|$ and $\Delta_N^u$; upper right: $e_N^{u^*}$, $\|R_N^*\|$ and $\Delta_N^{u^*}$; lower: $e_N^k$, $\eta_N^k$ and $\Delta_N^k$.
    }
    \label{fig:minicore_mean_errors_vs_APE_updated2}
\end{figure}

Finally, we gather in Table~\ref{table:cpu_times} the measured computational times for several quantities of interest and main stages in Python. Overall, the reduced basis method is very useful when the number $p$ of solutions that must be computed is very large, such as in an optimization process. Roughly, if $t_{\text{offline}}$ denotes the computational time of the \textit{offline} stage, $t_{\text{HF}}$ the \textit{high-fidelity} solver computational time, and $t_{\text{RB}}$ the reduced solver computational time, the reduced basis method becomes relevant when there holds
\begin{align*}
    t_{\text{offline}} + p \times t_{\text{RB}} < p \times t_{\text{HF}},
\end{align*}
that is
\begin{align*}
    p > \dfrac{t_{\text{offline}}}{t_{\text{HF}}-t_{\text{RB}}}.
\end{align*}
For this test case, this corresponds to $p > 1743$ parameter values.

\begin{table}[h!]
\centering
\begin{tabular}{|c|c|}
\hline
 & Mean computational time \\ \hline
\textit{Offline} stage & $\approx$ 11 hours \\ \hline
Assembling residual norm (\textit{offline} part) & 49.19 s \\ \hline
Assembling residual norm (\textit{online} part) & 5.03 s \\ \hline
Solving the \textit{high-fidelity} problem ($t_{\rm HF}$) & 14.71 s \\ \hline
Solving the reduced problem ($t_{\rm RB}$) & 0.44 s \\ \hline
\end{tabular}
\caption[Mean computational times for the Efficient Greedy reduced basis method applied to the 2D two-group \textit{minicore} in Python, for $N=100$]{Mean computational times for the Efficient Greedy reduced basis method applied to the 2D two-group \textit{minicore} in Python, for $N=100$}
\label{table:cpu_times}
\end{table}

\section*{Acknowledgements}

This project has received funding from the
European Research Council (ERC) under the European Union's Horizon 2020
research and innovation programme (grant agreement EMC2 No 810367). 
GD was supported by the French ‘Investissements d’Avenir’ program, project Agence
Nationale de la Recherche (ISITE-BFC) (contract ANR-15-IDEX-0003). GD was also supported by the Ecole des Ponts-ParisTech.

\bibliographystyle{siam}
\bibliography{biblio}

\end{document}